\documentclass[11pt]{amsart}

\usepackage{amsfonts}
\usepackage{graphicx}
\usepackage{hyperref}
\usepackage{xcolor}
\usepackage{amsmath, amsthm, amssymb, dsfont}
\usepackage{epsfig}
\usepackage{array}
\usepackage{amsrefs}
\usepackage{a4wide}
\usepackage{enumitem}
\usepackage{amssymb}
\usepackage{graphicx}
\usepackage{amsxtra}
\usepackage{mathbbol}
\usepackage{comment}
\usepackage{etoolbox}
\usepackage{amstext}
\usepackage{mathtools}
 \usepackage[normalem]{ulem}
\usepackage{microtype}

\DeclareSymbolFontAlphabet{\amsmathbb}{AMSb}

\numberwithin{equation}{section}

\newtheorem*{theorem*}{Theorem}
\newtheorem{theorem}{Theorem}[section]
\newtheorem{lemma}[theorem]{Lemma}
\newtheorem{remark}[theorem]{Remark}
\newtheorem{corollary}[theorem]{Corollary}
\newtheorem{proposition}[theorem]{Proposition}
\newtheorem{definition}[theorem]{Definition}
\newtheorem{example}[theorem]{Example}

\newcommand{\I}{\mathds{1}}
\newcommand{\st}{\;\vline\;}
\newcommand{\G}{\mathbb G}
\newcommand{\X}{\mathbb X}
\newcommand{\fF}{\mathcal F}
\newcommand{\h}{\mathbb H}

\newcommand{\field}[1]{\mathbb{#1}}
\newcommand{\C}{{\mathbb{C}}}

\newcommand{\la}{{\langle\,}}
\newcommand{\ra}{\,\rangle}
\newcommand{\id}{\mathrm{id}}
\newcommand{\ot}{{\, \otimes \, }}
\newcommand{\om}{{\omega}}

\newcommand{\vtp}{{\,\overline{\otimes}\,}}

\newcommand{\B}{{\mathcal{B}}}
\newcommand{\ad}{{\mathrm{ad}}}

\newcommand{\LL}{{L^{\infty}(\G)}}
\newcommand{\LO}{{{L}^{1}(\G)}}
\newcommand{\LT}{{L^{2}(\G)}}
\newcommand\numberthis{\addtocounter{equation}{1}\tag{\theequation}}

\newcommand{\PG}{{\mathcal{P}}}

\newcommand{\LLL}{{L^{\infty}(\widehat{\G})}}

\newcommand{\act}{\curvearrowright}
\newcommand{\R}{{\field{R}}}

\newcommand{\ww}{\mathrm{W}}
\newcommand{\WW}{{\mathds{V}\!\!\text{\reflectbox{$\mathds{V}$}}}}
\newcommand{\Ww}{\mathds{W}}
\newcommand{\wW}{\text{\reflectbox{$\Ww$}}\:\!}
\newcommand{\vv}{\mathrm{V}}

\newcommand{\norm}[1]{\|#1\|}

\newcommand{\fO}{\mathcal{O}}
\newcommand{\cst}{\ifmmode\mathrm{C}^*\else{$\mathrm{C}^*$}\fi}
\newcommand{\fP}{\mathcal{P}}
\newcommand{\hh}[1]{\widehat{#1}}

\newcommand{\crosGAr}
{\mathbb{G} \,\mbox{$_r$}\hspace{-.2ex}\mbox{$\ltimes$} \, A}

\newcommand{\crosGAF}
{\mathbb{G} \,\mbox{$_\mathcal{F}$}\hspace{-.2ex}\mbox{$\ltimes$} \, A}

\newcommand{\crosGGHr}
{\mathbb{G} \,\mbox{$_r$}\hspace{-.2ex}\mbox{$\ltimes$} \, {\B_{\Gamma_{\G/\h}}}}

\newcommand{\crosGGHF}
{\mathbb{G} \,\mbox{$_\mathcal{F}$}\hspace{-.2ex}\mbox{$\ltimes$} \, {\B_{\Gamma_{\G/\h}}}}

\newcommand{\crosGFF}
{\mathbb{G} \,\mbox{$_\mathcal{F}$}\hspace{-.2ex}\mbox{$\ltimes$} \, C(\partial_F\G)}
\newcommand{\crosGFr}
{\mathbb{G} \,\mbox{$_r$}\hspace{-.2ex}\mbox{$\ltimes$} \, C(\partial_F\G)}
\newcommand{\crosGBF}
{\mathbb{G} \,\mbox{$_\mathcal{F}$}\hspace{-.2ex}\mbox{$\ltimes$} \, \B_{\Gamma_\X}}
\newcommand{\crosGBr}
{\mathbb{G} \,\mbox{$_r$}\hspace{-.2ex}\mbox{$\ltimes$} \, \B_{\Gamma_{\X}}}

\newcommand{\fb}{{\partial_F(\G)}}

\DeclareMathOperator{\Mor}{Mor}
\DeclareMathOperator{\Image}{Image}
\DeclareMathOperator{\Span}{Span}
\DeclareMathOperator{\Fix}{Fix}
\def\Rep{\mathrm{Rep}}
\def\Irr{\mathrm{Irr}}

\def\Fix{\mathrm{Fix}}

\allowdisplaybreaks

\title[]{Topological Boundaries of Representations and Coideals}

\author[B. Anderson-Sackaney]{Benjamin Anderson-Sackaney}
\address{Benjamin Anderson-Sackaney\\
Department of Mathematics and Statistics\\
College of Arts and Science\\
University of Saskatchewan\\
Saskatoon S7N 5A2\\
CANADA
}
\email{fho407@usask.ca}

\author[F. Khosravi]{Fatemeh Khosravi}
\address{Fatemeh Khosravi\\
Department of Pure Mathematics\\
Faculty of Mathematics and Statistics\\
University of Isfahan\\
Isfahan 81746-73441\\
IRAN
}
\email{f.khosravi@mcs.ui.ac.ir}
\address{Fatemeh Khosravi\\
School of Mathematics\\
Institute for Research in Fundamental Sciences (IPM)\\
P. O. Box 19395-5746\\
Tehran\\
IRAN}
\email{f.khosravi@ipm.ir}


\thanks{{\ }\\[-3ex]B.A-S. was supported in part by the ANR (project ANR-19-CE40-0002), and in part by PIMS and the Simons Foundation (Simons grant PPTW GR023618).
\\ F.K. was  supported in part by a grant from IPM (No. 1402460318) and in part by Samsung Science and Technology Foundation under Project Number SSTF-BA2002-01 and the National Research Foundation of Korea (NRF) grant funded by the government of
Korea (MSIT) (No. 2017R1E1A1A03070510 and 2020R1C1C1A01009681). }

\begin{document}
\maketitle
\begin{abstract}
For a locally compact quantum group $\G$, a (left) coideal is a (left) $\G$-invariant von Neumann subalgebra of $L^\infty(\G)$. We introduce and analyze  various generalizations  of amenability and coamenability to coideals of discrete and compact quantum groups. We focus  on a particular class of coideals found in the category of compact quantum groups, which are associated with a compact quasi-subgroup. This class includes all coideals of the quotient type. We also introduce the notion of a  Furstenberg-Hamana boundary for representations of discrete quantum groups and use it to  study amenability and coamenability properties of coideals.
We then prove that a coideal of a compact quantum group that is associated with a compact quasi-subgroup is coamenable if and only if its codual coideal is $\G$-injective. If $\G$ is a unimodular or an exact discrete quantum group, we can replace $\G$-injectivity in the latter statement with the weaker condition of relative amenability. This result leads  to a complete characterization of the unique trace property. Specifically, a unimodular discrete quantum group $\G$ has the unique trace property if and only if the action of $\G$ on its noncommutative Furstenberg boundary is faithful. We also demonstrated that if  a unimodular discrete quantum group $\G$ is \cst-simple then it has the unique trace property. These findings are the quantum analogs of the groundbreaking results of Breuillard, Kalantar, Kennedy, and Ozawa and they provide answers to questions posted by Kalantar, Kasprzak, Skalski, and Vergnioux.

\end{abstract}

\section{Introduction}\label{sec1}
The Furstenberg boundary, denoted $\partial_F G$, is the unique universal $G$-boundary of a discrete group $G$ discovered by Furstenberg \cite{MR0146298}. A long-standing open problem in operator algebras had been the conjecture that the reduced group \cst-algebra $C^*_r(G)$ is simple if and only if $C^*_r(G)$ has the unique trace property. In 2017, Kalantar and Kennedy \cite{MR3652252} established the first link between the conjecture and the action of $G$ on its Furstenberg boundary by showing that $C^*_r(G)$ is simple if and only if the action of $G$ on its Furstenberg boundary is topologically free. Later, Breuillard, Kalantar, Kennedy, and Ozawa proved one directions of the conjecture by characterizing the unique trace property of $C^*_r(G)$ with faithfulness of the action of $G$ on its Furstenberg boundary. However, an example in \cite{MR3660307} disproved the other direction.

Kalantar and Kennedy proved that if $G$ is a discrete group, then $C(\partial_F G)$ is the $G$-injective envelope of $\C$ (as defined by Hamana). This identification, which was originally noted by Hamana himself without proof \cite{MR0509025}*{Remark 4}, can be used to construct a version of the Furstenberg boundary for discrete quantum groups. In fact, the non-commutative version of the Furstenberg boundary, $C(\fb)$, for a discrete quantum group $\G$, is the $\G$-injective envelope of $\C$ \cite{MR4442841}. The authors in \cite{MR4442841} also defined faithfulness of actions and showed that if the action of a discrete quantum group $\G$ on its Furstenberg boundary is faithful, then the Haar state is the only possible KMS-state for the scaling automorphism group – which, in the unimodular case, is the same as the uniqueness of trace. However, the converse was left unsolved. Additionally, the relationship between the \cst-simplicity of a discrete quantum group and the uniqueness of trace on $C(\hh\G)$ remained mysterious. Since it is still unknown how to define (topologically) free actions in the quantum setting, it is impossible to study \cst-simplicity in terms of the classical theory. Nevertheless, in this paper, we answer these questions elegantly by studying coideals and Furstenberg-Hamana boundaries of representations of discrete quantum groups.

In this paper, we focus on discrete quantum groups and compact quantum groups. In Section \ref{sec2}, we define these concepts and also introduce the theory of locally compact quantum groups and coideals. We then proceed to discuss the (right) coideals of a compact quantum group $\hh\G$ in different settings such as algebraic, \cst-algebraic, and von Neumann algebraic settings, which we denote respectively by $\fO(\hh\X\backslash\hh\G)$, $C(\hh\X\backslash\hh\G)$, and $L^\infty(\hh\X\backslash\hh\G)$. We also define an idempotent functional associated with these coideals, $\omega_{\hh\X}$, and show how using the duality of coideals, we can construct  (left) coideals in $\G$, denoted by $\ell^\infty(\X)$. It is worth noting that the coideals $\fO(\hh\X\backslash\hh\G)$ and $\ell^\infty(\X)$ are naturally corresponding to one another. We also introduce the concept of a compact quasi-subgroup of a compact quantum group, which is the most important notion we discuss. We say $\hh\X$ is a compact quasi-subgroup of $\hh\G$ if the idempotent functional $\omega_{\hh\X}$ associated with the (right) coideal $\fO(\hh\X\backslash\hh\G)$ is positive.

 By assuming that $\hh\X$ is a compact quasi-subgroup, we prove that coamenability of a coideal $\hh\X\backslash\hh\G$ in $\hh\G$ implies $\G$-injectivity of $\ell^\infty(\X)$, which in turn implies nuclearity of $C(\hh\X\backslash\hh\G)$. This result generalizes the fact that coamenability of $\hh\G$ implies amenability of $\G$ and nuclearity of $C(\hh\G)$.

In Section \ref{sec3}, we generalize some known results about amenability (coamenability) for discrete (compact) quantum groups to coideals. By assuming that $\hh\X$ is a compact quasi-subgroup, we  prove that coamenability of a coideal $\hh\X\backslash\hh\G$ in $\hh\G$ implies $\G$-injectivity of $\ell^\infty(\X)$, which in turn implies nuclearity of $C(\hh\X\backslash\hh\G)$. This result generalizes the fact that coamenability of $\hh\G$ implies amenability of $\G$ and nuclearity of $C(\hh\G)$ \cite{BT}.

In Section \ref{sec4}, we discuss the construction of Furstenberg-Hamana boundaries for representations of discrete quantum groups. This can be seen as a non-commutative version of the work carried out by Bearden and Kalantar, who studied representations of discrete groups \cite{MR4276323}. While constructing the Furstenberg-Hamana boundary for a representation based on \cite{MR4276323} may not be particularly surprising, the use of this notion of boundaries of representations to achieve the main objective of this paper is  noteworthy.

In Section \ref{sec5}, we delve into the topic of $\hh\G$-\cst-injectivity ($\hh\G$-$W^*$-injectivity) of $\hh\G$-\cst-operator systems ($\hh\G$-$W^*$-operator systems) for a compact quantum group $\hh\G$. A crucial finding in this section is that if the coideal $\ell^\infty(\X)$ of a compact quasi-subgroup $\hh\X$ of $\hh\G$ is $\G$-injective then $L^\infty(\hh\X\backslash\hh\G)$ is $\hh\G$-$W^*$-injective, which ultimately guarantees the injectivity of $L^\infty(\hh\X\backslash\hh\G)$.

In Section \ref{sec6}, we present a proof of a theorem that generalizes a well-known result about the equivalence of amenability of a discrete quantum group $\G$ with coamenability of its dual quantum group $\hh\G$ \cites{MR3693148,MR2276175} to coideals associated with compact-quasi subgroups. Specifically, we show that
\begin{theorem*}
    If $\hh\X$ is a compact quasi-subgroup of a compact quantum group $\hh\G$, then the coideal $\ell^\infty(\X)$ is $\G$-injective if and only if its codual coideal $\hh\X\backslash\hh\G$ is coamenable.
\end{theorem*}
Under the mild assumptions of unimodularity or exactness of a discrete quantum group $\G$, we can replace $\G$-injectivity with the weaker condition of relative amenability of $\ell^\infty(\X)$ in the latter theorem and still obtain coamenability of $\hh\X\backslash\hh\G$. This result has an important application, which is the main goal of this paper. Specifically,  if $\G$ is a unimodular discrete quantum group, we can complete the characterization of the unique trace property along with \cite{MR4442841}*{Theorem 5.3}:
\begin{theorem*}
    Let $\G$ be a unimodular discrete quantum group. Then $\G$ has the unique trace property iff the action of $\G$ on its Furstenberg boundary is faithful.
\end{theorem*}
The latter theorem leads to a generalization of a famous result in \cite{MR3735864} for unimodular discrete quantum groups.
\begin{theorem*}
    Let $\G$ be a unimodular discrete quantum group. If $\G$ is \cst-simple then $\G$ has the unique trace property.
\end{theorem*}
Finally, we provide examples of discrete quantum groups that satisfy both \cst-simplicity and exactness. Our findings indicate that these quantum groups act faithfully on their Furstenberg boundaries, and in unimodular cases, they possess the unique trace property.
These examples include the duals of certain quantum automorphism groups, free unitary and free orthogonal quantum groups. They are well-known in the literature and have attracted considerable attention for the study of their operator algebraic properties. Prior to our work, the faithfulness of the actions of some of these examples on their Furstenberg boundaries, as well as the unique trace property in unimodular cases, were investigated independently using various methods.

\section{Preliminaries}\label{sec2}
This section aims to establish definitions, fix notations, and provide some references. Throughout the paper, the flip morphism between tensor products of operator algebras will be denoted by the symbol $\sigma$. The state space of a \cst-algebra $A$ will be represented by $S(A)$. The set of non-degenerate $*$-homomorphisms from $A$ to the multiplier algebra of $B$, $M(B)$ will be denoted by $\mathrm{Mor}(A,B)$, where $A$ and $B$ are \cst-algebras.
Various tensor products, such as algebraic tensor products, min tensor products of \cst-algebras, and Hilbert space tensor products, will be represented by the symbol $\otimes$, with the exact meaning depending on the context. The von Neumann tensor product will be represented by $\vtp$. The commutant and predual of a von Neumann algebra $M$ will be represented by $M^\prime$ and $M_*$, respectively. We will sometimes consider an element $x\in M\vtp N$, where $M$ and $N$ are von Neumann algebras, and slice this element by a non-normal functional $\psi\in M^*$. Note that $y=(\psi\otimes \id)x$ is an element of $N$ such that for every normal functional $\phi\in  N_*$, $\phi(y)= \psi((\id\otimes\phi)x)$.
For an element $x\in A\otimes A$, we use the leg numbering notation $x_{12}=x\otimes 1,x_{13}=(\id\otimes \sigma)x_{12},x_{23}=1\otimes x $ as elements in $A\otimes A\otimes A$. We let $B(H)$ denote the set of all bounded operators on a Hilbert space $H$. For $\eta,\xi\in H$, the symbol $\omega_{\eta,\xi}\in B(H)_*$ is the functional $T\mapsto \langle T\eta ,\xi\rangle$, with $\omega_{\xi}:=\omega_{\xi,\xi}$ .

We briefly introduce operator systems and two tensor products on operator systems and refer the reader to \cites{MR1976867, MR1793753} for more details. Operator systems are unital self-adjoint subspaces of $B(H)$. The category of operator systems is equipped with unital completely positive (u.c.p.) maps as morphisms.  

Given operator systems $X\subseteq B(H)$ and $Y\subseteq B(K)$, the spatial tensor product is the norm closure of the algebraic tensor product $X\otimes Y$ in $B(H\otimes K)$ and will be denoted by $\otimes_{sp}$.
In the case where $X$ and  $Y$ are \cst-algebras, the spatial tensor product is exactly the min tensor product of \cst-algebras. The Fubini tensor product of operator systems is defined as follows:
\[
X\otimes_{\fF} Y = \{Z\in B(H\otimes K) : (\varphi\otimes\id)(Z) \in Y,~ (\id\otimes \psi)(Z)\in X, ~\varphi \in B(H)_*, \psi\in B(K)_*\}.
\]
When $X$ and $Y$ are von Neumann algebras, the Fubini tensor product is exactly the von Neumann tensor product.

\subsection{Discrete and compact quantum groups}
Our principal references for the fundamentals of the theory of quantum groups are \cites{MR1832993, MR1951446, Ne, MR1616348}. A von Neumann algebraic locally compact quantum group is a quadruple $\mathbb{G} = (\LL, \Delta_\G,\varphi_\G,\psi_\G)$, where $\LL$ is a von Neumann algebra with a coassociative comultiplication $\Delta_\G\colon\LL\to\LL\vtp\LL$, and $\varphi_\G$ and $\psi_\G$ are, respectively, normal semi-finite faithful (n.s.f.) left and right Haar weights on $\LL$.
The GNS Hilbert space of the left Haar weight $\varphi_\G$ will be denoted by  $\LT$ and we will write $\LO$ for the predual of the von Neumann algebra $\LL$. The antipode, the scaling group and the
unitary antipode of a quantum group $\G$ will be denoted by $S_\G$, $(\tau_t^\G)_{t\in \mathbb{R}}$ and $R_\G$. 

We denote by $\ww_\G,\vv_\G\in B(\LT\otimes \LT)$ the left and right  multiplicative unitaries of $\G$, respectively which  implement the comultiplication: $\Delta_\G(x)=\vv_\G(x\otimes 1)\vv^*_{\G}=\ww^*_{\G}(1\otimes x)\ww_\G$ for $x\in\LL$. Moreover,
$\LL =\bigl\{ (\omega\otimes\id)\vv_\G\st\omega\in B(\LT)_*\bigr\} ^{\prime\prime}=\bigl\{ (\id\otimes\omega)\ww_\G\st\omega\in B(\LT)_*\bigr\}^{\prime\prime}$. For a locally compact quantum group $\G$, $C_0(\G)$ and $C_0^u(\G)$ will denote the reduced and universal \cst-algebras associated with $\G$. Since $C_0(\G) \subseteq L^\infty(\G)$ we will use $\Delta_\G$ to also denote the comultiplication on $C_0(\G)$, so we have $\Delta_\G\in \Mor(C_0(\G), C_0(\G) \otimes C_0(\G))$, while the comultiplication on $C^u_0(\G)$ will be denoted by $\Delta_{\G}^u$.
There is a canonical surjective reducing morphism $\Lambda_\G:C_0^u(\G)\to C_0(\G)$, intertwining the
respective comultiplications.
The counit $\varepsilon_\G:C_0^u(\G)\to \C$ is  a $*$-homomorphism that satisfies $(\id\otimes\varepsilon_\G)\Delta^u_\G=\id= (\varepsilon_\G\otimes\id)\Delta^u_\G$. A locally compact quantum group is called {\it coamenable} if the reduced version of counit exists which is equivalent to $\Lambda_\G$ being injective.

The dual of a locally compact quantum group $\G$ will be denoted by $\hh\G$. We have
$L^\infty(\hh\G)=\bigl\{( \id\otimes\omega)\ww_{\hh\G}\st\omega\in B(\LT)_*\bigr\}^{\prime\prime}$, $\Delta_{\hh\G}(x)=\ww_{\hh\G}^*(1\otimes  x){\ww_{\hh\G}}$ for $ x\in\LLL$,
where $\ww_{\hh\G}:=\sigma({\ww_{\G}})^*$. We have $\ww_\G\in  \LL\vtp L^\infty(\hh\G)$ and $\vv_\G\in L^\infty(\hh\G)^\prime\vtp\LL$. The half-lifted and universal versions of multiplicative unitaries exist. In particular, for the left multiplicative unitary $\ww_\G$, we have $\WW_\G\in M(C_0^u(\G)\otimes C_0^u(\hh\G))$, $\wW_\G\in M(C_0(\G)\otimes C_0^u(\hh\G))$, and $\Ww_\G\in M(C^u_0(\G)\otimes C_0(\hh\G))$ such that  
\[
\ww_\G= (\Lambda_\G\otimes \Lambda_{\hh\G})\WW_\G= (\id\otimes \Lambda_{\hh\G})\wW_\G= (\Lambda_\G\otimes \id)\Ww_\G.
\]
A locally compact quantum group $\G$ is \emph{compact} if the Haar weights are finite and is \emph{discrete} if its dual $\hh\G$ is compact. In this paper, $\G$ will usually denote a discrete quantum group and $\hh\G$ will denote a compact quantum group. When $\hh\G$ is compact, $\varphi_{\hh\G} = \psi_{\hh\G} = h_{\hh\G}$.
A compact quantum group $\hh\G$ is \emph{Kac type} if $h_{\hh\G}$ is a trace. It is known that  a compact quantum group $\hh\G$ is Kac type iff its dual discrete quantum group $\G$ is unimodular, which means the left and right Haar weights can be scaled to be equal.

Let $\hh\G$ be a compact quantum group.
We denote the set of all equivalent classes of finite-dimensional unitary corepresentations  of $\hh\G$ by $\Rep(\hh\G)$.
The linear span of matrix coefficients of all finite-dimensional corepresentations of $\hh\G$, is a dense unital $*$-subalgebras of both $C(\hh\G)$ and $C^u(\hh\G)$, denoted by $\fO(\hh\G)$,
 \[\fO(\hh\G) =\{u_{\xi,\eta}:= (id \otimes \omega_{\eta,\xi})(u) : u\in \Rep(\hh\G), \xi,\eta\in H_u\}.\]
 Moreover, $\fO(\hh\G)$ is a Hopf $*$-algebra with the restrictions of the comultiplication $\Delta_{\hh\G}$, the  Haar state $h_{\hh\G}$, the antipode $S_{\hh\G}$, and the counit $\varepsilon_{\hh\G}$, with $\fO(\hh\G)'' =L^\infty(\hh\G)\subseteq B(L^2(\G))$.
 
 We also denote the set of all equivalence classes of  irreducible unitary corepresentations by $\Irr(\hh\G)$.
We let $c_{00}(\G) := \oplus_{u\in \Irr(\hh\G)} B(H_u)$ and $c_0(\G) = \overline{c_{00}(\G)}$. Then $\ell^\infty(\G) := c_{00}(\G)''$ is the locally compact quantum group associated with $\G$ in the sense of Kustermans and Vaes \cite{MR1832993}. A discrete quantum group $\G$ is always coamenable with a normal counit, $\varepsilon_\G\in \ell^1(\G)$.

We set
\[
     \ell^1_F(\G) := \oplus_{u\in \Irr(\hh\G)} B(H_u)_*.
\]
The dual space of any coalgebra $(A,\Delta)$ is an algebra via convolution: $\mu*\nu = (\mu\otimes \nu)\Delta$. In particular, $\ell^1(\G)$ and $\fO(\hh\G)^*$ are algebras with convolution. Then, the maps
$$\{h_{\hh\G}(a\cdot) : a\in \fO(\hh\G)\}\to c_{00}(\G), \mu \mapsto  (\mu\otimes\id)(\ww_{\hh\G})$$
and
$$\ell^1_F(\G) \to \fO(\hh\G), ~ f\mapsto (f\otimes\id)(\ww_\G).$$
are algebra isomorphisms. We will freely use the algebra isomorphism $\ell^1_F(\G)\cong \fO(\hh\G)$.

\vspace{.2cm}

Let $\G$ be a discrete quantum group. A $\G$-representation is a unital $*$-representation $\pi : \fO(\hh\G)\to B(H_\pi)$. We set $C^\pi(\hh\G) = \overline{\pi(\fO(\hh\G))}\subseteq B(H_\pi)$.
In particular, the reduced and universal compact quantum groups  \cst-algebras are respectively, $C(\hh\G):=C^{\lambda_\G}(\hh\G)$ and $C^u(\hh\G):=C^{\varpi}(\hh\G)$, where $\lambda_\G : \fO(\hh\G)\to B(\ell^2(\G))$ and $\varpi : \fO(\hh\G) \to B(H_\varpi)$ are respectively the left regular representation and the universal representation of $\G$.
Using the universal property of $\varpi$, every $\G$-representation $\pi$ extends continuously to a unital $*$-representation $\pi : C^u(\hh\G)\to B(H_\pi)$ such that $\varpi(a)\mapsto \pi(a)$, $a\in \fO(\hh\G)$.
Recall that the reducing morphism  $\Lambda_{\hh\G}:C^u(\hh\G)\to C(\hh\G)$ is the extension of the left regular representation,  $\lambda_{\G}$.
It is known that there is a bijection between $\G$-representations $\pi$ and unitary corepresentations, unitary operators $U_\pi\in \ell^\infty(\G)\overline{\otimes}B(H_\pi)$ such that $(\Delta_\G\otimes\id)(U_\pi) = (U_\pi)_{13}(U_\pi)_{23}$, where  $U_\pi = (\id\otimes \pi)(\wW_\G)$.

\vspace{.2cm}

There are two definitions of closed quantum subgroups of a locally compact quantum group given by Woronowicz and Vaes.  It was shown with \cite{MR2980506}*{Theorem 3.5, Theorem 6.1, Theorem 6.2} that for compact and discrete quantum groups, these two definitions are equivalent.
\begin{definition}\cite{MR2980506}\label{Vcqg}
    Let $\G$ and $\h$ be locally compact quantum groups. We say that $\h$ is a closed quantum subgroup of $\G$,
    \begin{itemize}
        \item (Woronowicz) if there is a morphism $\pi\in \Mor(C_0^u(\G), C_0^u(\h))$ that intertwines the comultiplications, i.e. $(\pi\otimes\pi)\circ\Delta^u_\G=\Delta^u_\h\circ \pi$, and $\pi(C_0^u(\G))=C_0^u(\h)$.
        \item (Vaes) if there is an injective unital normal $*$-homomorphism $\gamma: L^\infty(\hh\h)\to L^\infty(\hh\G)$ that intertwines the comultiplication, i.e. $(\gamma\otimes\gamma)\circ\Delta_{\hh\h} = \Delta_{\hh\G}\circ\gamma.$
    \end{itemize}
\end{definition}
Let $\G$ be a locally compact quantum group. A (left) action of $\G$ on a 
\cst-algebra $A$ is an injective morphism $\alpha: A\to M(C_0(\G)\otimes A)$ such that $(\Delta_\G\otimes \id)\circ \alpha = (\id\otimes\alpha)\circ \alpha$. The action $\alpha$ is said to be continuous if $\alpha(A)(C_0(\G)\otimes 1) = C_0(\G) \otimes A$. It turns out that any action of a discrete quantum group $\G$ on a unital \cst-algebra $A$ is
automatically continuous (see e.g. the last part of the proof of \cite{MR4442841}*{ Theorem 4.9}).

A (left) action of a locally compact quantum group $\G$ on a von Neumann algebra $M$ is a unital injective $*$-homomorphism $\alpha: M\to L^\infty(\G)\vtp M$ such that $(\Delta_\G\otimes \id)\circ \alpha = (\id\otimes\alpha)\circ \alpha$. The fixed point algebra of the action $\alpha$ is a von Neumann subalgebra of $M$ defined by $\Fix(\alpha)=\{x\in M \st \alpha(x)=1\otimes x\}.$ %Let $\h$ be a (Vaes) closed quantum subgroup of $\G$ associated to a morphism $\gamma$, then there is a (left) action of $\h$ on $L^\infty(\G)$, $\alpha_{\h}(x)= V^*(1\otimes x)V$, where $V=(\id\otimes \gamma)\ww_\h\in L^\infty(\h)\vtp L^\infty(\hh\G)$. The fixed point algebra of $\alpha_\h$ will be denoted by $L^\infty(\h\backslash \G)$.

Let $\h$ be a (Vaes) closed quantum subgroup of $\G$ with associated morphism $\gamma$. Then there are \cst-algebraic and von Neumann algebraic (left) actions of $\h$  on $C_0(\G)$ and $L^\infty(\G)$, respectively given by the morphism $U_{\h}^* (\I\otimes \cdot) U_{\h}$,
where $U_{\h}= (\id\otimes\gamma)\ww_\h$. 
Assume  $\h$ is a quantum subgroup of a compact quantum group $\G$, the  fixed point algebras of these actions will be denoted  by $C(\h\backslash\G)$ and $L^\infty(\h\backslash\G)$, respectively. 

\subsection{Coideals of discrete and compact quantum groups} We introduce the notion of coideals in the von Neumann algebraic setting for locally compact quantum groups. Meanwhile, we are dealing with coideals of compact and discrete quantum groups.
\begin{definition}
Let $\G$ be a locally compact quantum group. A von Neumann subalgebra $N\subseteq L^\infty(\G)$ is called
\begin{itemize}
    \item {\it (Right) coideal} if $\Delta_\G(N)\subset N\vtp L^\infty(\G)$;
   \item {\it Integrable coideal} if $N$ is a (right) coideal and $\varphi_{\G}|_{N}$ is semifinite;     \item {\it Invariant subalgebra} if $\Delta_\G(N)\subseteq N\vtp N$;
    \item {\it Baaj-Vaes subalgebra} if $N$ is an invariant subalgebra of $L^\infty(\G)$ which is preserved by the unitary antipode $R_\G$ and the scaling group $(\tau_t^\G)_{t\in \R}$ of $\G$.
\end{itemize}
\end{definition}
There is a natural coduality between (right) coideals of $\G$ and $\hh\G$. In particular, if $N\subseteq L^\infty(\G)$ is a (right) coideal then its codual $\tilde{N}:= N^\prime \cap L^\infty(\hh\G)$ is a (right) coideal of $\hh\G$ and $\tilde{\tilde{N}}=N$. If we consider a closed quantum subgroup $\h$ of $\G$, then the fixed point algebra $L^\infty(\h\backslash\G)$ is a (right) coideal with $L^\infty(\hh\h)$ as its codual. 
 Note that a (right) coideal $N\subseteq L^\infty(\G)$  is called of {\it quotient type} if there exists a closed quantum subgroup $\h \leq \G$ such that $N=L^\infty(\h\backslash\G)$.

For discrete and compact quantum groups it was proved that every invariant subalgebra is a Baaj-Vaes subalgebra with the results in \cite{MR3291643}*{Theorem 3.1}. Moreover, by \cite{MR2115071}*{Proposition 10.5} and Definition \ref{Vcqg}, there is a one-to-one correspondence between closed quantum subgroups of $\G$ and Baaj-Vaes subalgebras of $L^\infty(\hh\G)$. 

\vspace{0.5cm}

The term ``coideal \cst-subalgebra of quantum groups" is not commonly used in the literature. However, for compact quantum groups $\hh\G$, since the \cst-algebras are unital, we will use this terminology.
\begin{definition}
    Let $\hh\G$ be a compact quantum group. A \cst-subalgebra $A\subseteq C(\hh\G)$ (respectively $A\subseteq C^u(\hh\G)$) is called (right) coideal \cst-subalgebra if 
    $\Delta_{\hh\G}(A)\subseteq A\otimes C(\hh\G)~(\text{respectively}~\Delta^u_{\hh\G}(A)\subseteq A\otimes C^u(\hh\G))$.
 \end{definition}
%\vspace{0.3cm}
We briefly introduce algebraic coideals of the Hopf $*$-algebra $\fO(\hh\G)$ of a compact quantum group $\hh\G$ and refer the reader to \cite{MR3863479} for more details.

\begin{definition}
    A (right) {\it coideal subalgebra} of $\hh\G$ is a unital $*$-subalgebra $A\subseteq \fO(\hh\G)$ such that
    $\Delta_{\hh\G}(A)\subseteq A\otimes \fO(\hh\G)$.
\end{definition}

Let $\hh\h$ be a closed quantum subgroup of a compact quantum group $\hh\G$. Then, there exists a surjective  unital $*$-homomorphism $\pi_{\hh\h}:\fO(\hh\G)\to \fO(\hh\h)$ such that $(\pi_{\hh\h}\otimes \pi_{\hh\h})\Delta_{\hh\G}= \Delta_{\hh\h}\circ \pi_{\hh\h}$. Let us denote
\[
\fO(\hh\h\backslash\hh\G):=\{x\in \fO(\hh\G)\st (\pi_{\hh\h}\otimes \id)\Delta_{\hh\G}(x)=\pi_{\hh\h}(1)\otimes x \}. 
\]
Then $\fO(\hh\h\backslash\hh\G)$ is a (right) coideal subalgebra of $\hh\G$. Let $A$ be a (right) coideal subalgebra of $\hh\G$, we say $A$ is of {\it quotient type} if there exists a closed quantum subgroup $\hh\h\leq\hh\G$ such that $A = \fO(\hh\h\backslash\hh\G)$.
%We refer the reader to \cite{} for more details about algebraic coideals of Hopf $*$-algebras. Let $\hh\G$ be a compact quantum group with associated
%Hopf $*$-algebra $\fO(\hh\G)$.
%\begin{definition}
%    A (right) {\it coideal subalgebra} of $\hh\G$ is a unital $*$-subalgebra $A\subseteq \fO(\hh\G)$ such that
 %   $\Delta_{\hh\G}(A)\subseteq A\otimes \fO(\hh\G)$. We say $A$ is of {\it quotient type} if there exists a closed quantum subgroup $\hh\h\leq\hh\G$ such that $A = \fO(\hh\h\backslash\hh\G)$.
%\end{definition}
It is known a coideal subalgebras of $\hh\G$ is not necessarily of  quotient type  but it can be stated as a  quotient of a quantum space,  \cite{MR3863479}.  We briefly explain this fact by setting up the bijection between (right) coideal subalgebras of $\fO(\hh\G)$  and left $\fO(\hh\G)$-module $*$-coalgebra quotients of $\fO(\hh\G)$ as established in \cite[Theorem 3.2]{MR3863479}. %We briefly explain this bijection for the Hopf $*$-algebra $\fO(\hh\G)$. % and leave the details to the reader to see more in \cite{C18} and also in \cite{DCDT21}.
Given a coideal subalgebra $A\subseteq \fO(\hh\G)$, set $C := \fO(\hh\G)/\fO(\hh\G)A_+$ where $A_+ = A\cap \ker(\varepsilon_{\hh\G})$. By \cite[Proposition 1]{Tak79} $C$ has a unique structure of a left $\fO(\hh\G)$-module coalgebra such that the projection map $\pi_C: \fO(\hh\G) \to C$ is a left $\fO(\hh\G)$-module coalgebra map. By \cite[Theorem 3.1]{MR3863479} $C$ is cosemisimple, so we can obtain
$A$ back from $C$ via
\[
A= \{ x \in \fO(\hh\G)\st (\pi_C\otimes \id)\Delta_{\hh\G}(x) = \pi_C(1)\otimes x\}.
\]
 In light of this discussion, when considering  a right coideal subalgebra $A$ of $\hh\G$, we will always invoke the ``quantum space'' notation $\hh\X$  and express $A = \fO(\hh\X\backslash\hh\G)$ where $\fO(\hh\X)$ represents the corresponding left $\fO(\hh\G)$-module $*$-coalgebra with the coproduct denoted by $\Delta_{\hh\X}$. Let $q_{\hh\X} : \fO(\hh\G)\to \fO(\hh\X)$ be the quotient map, which is a surjective linear map that satisfies $(q_{\hh\X}\otimes q_{\hh\X})\circ\Delta_{\hh\G} = \Delta_{\hh\X}\circ q_{\hh\X}$ (note that in terms of the notation above, we are setting $q_{\hh\X} := \pi_{\fO(\hh\X)}$). Then
$$\fO(\hh\X\backslash\hh\G) := \{a\in \fO(\hh\G) : (q_{\hh\X}\otimes\id)\Delta_{\hh\G}(a) = q_{\hh\X}(1)\otimes a\} = A.$$

%Then the quotient map $\pi: \fO(\hh\G)\to C$ induces a coalgebra structure on $C$ so $C$ is a left $\fO(\hh\G)$-module $*$-coalgebra. 
%set $\fO(\hh\X) := \fO(\hh\G)/\fO(\hh\G)A_+$ where $A_+ = A\cap \ker(\epsilon_{\hh\G})$. The space $\fO(\hh\X)$ is a coalgebra, and we denote its coproduct by $\Delta_{\hh\X}$. Let $q_{\hh\X} : \fO(\hh\G)\to \fO(\hh\X)$ be the quotient map, which satisfies $(q_{\hh\X}\otimes q_{\hh\X})\Delta_{\hh\G} = \Delta_{\hh\X}\circ q_{\hh\X}$. Then
%$$\fO(\hh\X\backslash\hh\G) := \{a\in \fO(\hh\G) : (q_{\hh\X}\otimes\id)\Delta_{\hh\G}(a) = q_{\hh\X}(1)\otimes a\} = A.$$
%In light of all this, for any coideal $A$ we invoke $\X$ above and write $A = \fO(\hh\X\backslash\hh\G)$.

Since $(\fO(\hh\X), \Delta_{\hh\X})$ is a coalgebra, the dual space $\fO(\hh\X)^*$ is an algebra via convolution: $\mu*\nu = (\mu\otimes\nu)\Delta_{\hh\X}$. Let $h_{\hh\X} \in \fO(\hh\X)^*$ be the self-adjoint support projection of the following representation,
\[
\fO(\hh\X)^*\to \C, \quad \mu\mapsto \mu(1_{\hh\X}),
\]
where $1_{\hh\X}= q_{\hh\X}(1)$. Then, the functional  $h_{\hh\X}: \fO(\hh\X)\to \C$ (not necessarily positive) has the invariance property, \cite{DCDT21}
\[
(h_{\hh\X}\otimes\id)\Delta_{\hh\X} = 1_{\hh\X} h_{\hh\X}  = (\id\otimes h_{\hh\X})\Delta_{\hh\X}.
\]

%The functional $h_{\hh\X}$ is not necessarily positive such that $h_{\hh\X}*\mu = h_{\hh\X} = \mu* h_{\hh\X}$ for every $\mu\in \fO(\hh\X)^*$. 
Moreover,  $\omega_{\hh\X}:= h_{\hh\X}\circ q_{\hh\X} \in \fO(\hh\G)^*$ is a unital idempotent functional which is not necessarily positive.

Note that for any coideal subalgebra $\fO(\hh\X\backslash\hh\G)\subseteq \fO(\hh\G)$, we can construct operator (\cst  or von Neumann) algebraic coideals as follows. Let us denote
\begin{itemize}
\item $\fO(\hh\X\backslash\hh\G)''$ by $L^\infty(\hh\X\backslash\hh\G)$ which is a coideal von Neumann subalgebra of $L^\infty(\hh\G)$;
    \item the norm closure of $\fO(\hh\X\backslash\hh\G)$  in $C(\hh\G)$ by $C(\hh\X\backslash\hh\G)$ which is a coideal \cst-subalgebra of $C(\hh\G)$;
    \item universal \cst-norm closure of $\fO(\hh\X\backslash\hh\G)$  in $C^u(\hh\G)$ by $C^u(\hh\X\backslash\hh\G)$ which is a coideal \cst-subalgebra of $C^u(\hh\G)$.
\end{itemize}
On the other hand, having an operator (\cst- or von Neumann) algebraic coideal $A$ of $\hh\G$, $A\cap \fO(\hh\G)$ is an algebraic coideal (coideal subalgebra) of $\hh\G$.  This is a one-to-one correspondence between algebraic and operator algebraic coideals of $\hh\G$.  Therefore, for a compact quantum group $\hh\G$, with the notation $\hh\X\backslash\hh\G$ we can refer to any of the (right) algebraic or operator algebraic concepts of coideals. 

Given a (right) coideal $L^\infty(\hh\X\backslash\hh\G)\subseteq L^\infty(\hh\G)$, we denote its (right) codual coideal by $\ell^\infty(\X_r)$,
\[
\ell^\infty(\X_r):=\widetilde{L^\infty(\hh\X\backslash\hh\G)} = L^\infty(\hh\X\backslash\hh\G)'\cap\ell^\infty(\G) \subseteq \ell^\infty(\G)
\]
 We adopt this notation from the situation where $L^\infty(\hh\X\backslash\hh\G)$ is of quotient type, highlighting that the codual is a (right) coideal. Note that all coideals of compact quantum groups are integrable. Let $L^2(\hh\X\backslash\hh\G)$ denote the closed subspace of $L^2(\hh\G)$ generated by $L^\infty(\hh\X\backslash\hh\G)$ via the GNS construction of $h_{\hh\G}$. There exists an orthogonal projection $Q_{\X} : L^2(\hh\G) \to L^2(\hh\X\backslash\hh\G)$. It has been shown that $Q_{\X}$
 is a group-like projection in $ \ell^\infty(\G)$ \cite[Lemma 3.4, Lemma 3.5]{MR2335776},  a self-adjoint projection  that satisfies
$(Q_{\X}\otimes \I)\Delta_\G(Q_\X) = Q_{\X}\otimes Q_{\X}$. Moreover, $Q_{\X}$ is a minimal central projection in $\ell^\infty(\X_r)$ that can reconstruct the coideal $\hh\X\backslash\hh\G$ and its codual  \cite{MR3894765}*{Proposition 1.5, Corollary 1.6},
\begin{align*}
      \fO(\hh\X\backslash\hh\G) &=\{ (\id\otimes   Q_{\X}\omega)(\ww^{\hh\G})^* \st \omega\in \ell_{F}^1(\G) \}, \numberthis \label{algquo}
      \\
      L^\infty(\hh\X\backslash\hh\G) &= \{  (\id\otimes   Q_{\X}\omega)(\ww^{\hh\G})^* \st \omega\in \ell^1(\G)  \}^{\sigma-\text{weak closure}},
      \\
   \ell^\infty(\X_r) &= \{ x\in \ell^\infty(\G) \st (Q_{\X}\otimes 1)\Delta_\G(x) = Q_{\X}\otimes x \}.
\end{align*}
In \cite{MR3894765}, the duality is established through  the right regular representation (along with some additional adjustments on the duality), whereas we utilize the left regular representation. It is evident  that by employing the unitary antipode of $\G$, we can transition from (right) coideals to (left) coideals or vice versa. Hence, for each (left) coideal of $\G$, denoted by $\ell^\infty(\X)$, there exists a (right) coideal $\ell^\infty(\X_r)$ such that $\ell^\infty(\X)= R_\G (\ell^\infty(\X_r))$. It is notable that $P_{\X} = R_\G(Q_{\X})$ is a group-like projection, $\Delta_{\G}(P_{\X})(\I\otimes P_{\X})= P_{\X}\otimes P_{\X}$.  Particularly, for every (left) coideal of $\G$, we have
\begin{equation}\label{linftyX}
    \ell^\infty(\X)= R_\G\left(\ell^\infty(\X_r)\right)=\{x\in \ell^\infty(\G)\st \Delta_\G(x)(\I\otimes P_{\X})= x\otimes P_{\X} \}.
\end{equation}
Using \eqref{linftyX}, we observe that $P_\X x = \varepsilon_\G(x)P_\X$ for all $x\in\ell^\infty(\X)$.

%$$\widetilde{L^\infty(\hh\X\backslash\hh\G)} = L^\infty(\hh\X\backslash\hh\G)'\cap\ell^\infty(\G) \subseteq \ell^\infty(\G)$$
%is a right coideal in $\G$, which is called the {\it codual coideal} \cite{MR2335776}. Note that for a (right) coideal $L^\infty(\hh\X\backslash\hh\G)$, we will denote its (right) codual by $\ell^\infty(\X_r)$. We borrow this notation from  the situation where $L^\infty(\hh\X\backslash\hh\G)$ is quotient type.
%\vspace*{2cm}

%The algebra $\{h_{\hh\X}(a\cdot) : a\in \fO(\hh\G)\}$ identifies as a $*$-subalgebra of $\ell^\infty(\G)$. Let $\fF : \{h_{\hh\X}(a\cdot) : a\in \fO(\hh\G)\} \to \ell^\infty(\G)$ denote this inclusion and set $c_{00}(\X) = \Image(\fF)$. We let
%$$\ell^\infty(\X) := c_{00}(\X)'' \subseteq \ell^\infty(\G)$$.
%Since  $\fO(\hh\X)$ is a left $\fO(\hh\G)$-module so $\ell^\infty(\X)$ is a coideal of $\G$. Furthermore, $\fF(h_{\hh\X}) = P_\X \in l^\infty(\G)$ is a self-adjoint projection that satisfies
%$$(\I\otimes P_\X)\Delta_\G(P_\X) = P_\X\otimes P_\X$$
%and
%$$\ell^\infty(\X) = \{x\in \ell^\infty(\G) : (\I\otimes P_\X)\Delta_\G(x) = x\otimes P_\X\}.$$
%Such a projection $P_\X$ is known as a \textcolor{red}{{\it  group-like projection}}. In light of all this, for any coideal $N$ of $\G$ we will invoke $\X$ above and write $N = \ell^\infty(\X)$. Note that when $\hh\h \leq \hh\G$, $\Delta_{\h} = \Delta_{\G}|_{\ell^\infty(\h)}$.

Finally, It should be noted that we can reconstruct the (right) coideal $\hh\X\backslash\hh\G$ by $P_\X$, the group-like projection associated with the (left) coideal $\ell^\infty(\X)$,
 $$\fO(\hh\X\backslash\hh\G) = \{u_{P_\X\xi,\eta} : u\in \Rep(\hh\G), \xi,\eta\in H_u\} = \{ (P_\X \omega\otimes\id)(W_\G) : \omega\in \ell^1_F(\G)\}.$$

\begin{remark}\label{Coideals of Quantum Groups}
     Let us conclude this subsection by summarizing the main point we want to emphasize about coideals of discrete and compact quantum groups. There is a bijection between the following sets of  objects:
\begin{itemize}
     \item (right) coideal subalgebras $\fO(\hh\X\backslash\hh\G)\subseteq \fO(\hh\G)$;
     \item (right)  coideal \cst-subalgebras $C(\hh\X\backslash\hh\G)\subseteq C(\hh\X\backslash\hh\G)$;
     \item (right)  coideal \cst-subalgebras $C^u(\hh\X\backslash\hh\G)\subseteq C^u(\hh\X\backslash\hh\G)$;
     \item (right)  coideal von Neumann subalgebras $L^\infty(\hh\X\backslash\hh\G)\subseteq L^\infty(\hh\G)$;
     \item (right) coideal von Neumann subalgebras $\ell^\infty(\X_r)\subseteq \ell^\infty(\G)$;
     \item (left) coideal von Neumann subalgebras $\ell^\infty(\X)\subseteq\ell^\infty(\G)$;
     \item group-like projections $P_\X\in \ell^\infty(\G)$.
\end{itemize}
  In particular, when we invoke any of the above objects, the existence of each of the other objects is automatically implied. We will frequently use this fact.
\end{remark}

\subsection{Actions of discrete quantum groups}\label{2.4}
We will start with the definition of actions of discrete quantum groups on operator systems given in \cite{dRH23}.
\begin{definition}
Let $\G$ be a discrete quantum group.   An operator system $X$ is called a (left) $\G$-operator system if there exists a unital completely isometric (u.c.i.) map $\alpha : X\to \ell^\infty(\G)\otimes_{\fF} X$ such that $(\Delta_\G\otimes\id)\circ\alpha= (\id\otimes\alpha)\circ\alpha$ and
    $$\overline{\Span}(c_0(\G)\otimes 1)\alpha(X) = c_0(\G)\otimes X.$$
    The map $\alpha$ is called the action of $\G$ on $X$ and will denote by $\alpha: \G\act X$.
   Note that, if $X$ is a unital \cst-algebra then $\alpha$ is a $*$-homomorphism and we call $X$ a $\G$-\cst-algebra.
Let $X$ and $Y$ be $\G$-operator systems with corresponding actions $\alpha$ and $\beta$ respectively. Then a u.c.p. map $\Psi : X\to Y$ is called {\it $\G$-equivariant} if $\beta\circ\Psi = (\id\otimes\Psi)\circ\alpha$.
\end{definition}
Let $A$ be a $\G$-\cst-algebra, $\alpha:\G\act A$, and $\mu\in S(A)$, the $\G$-equivariant u.c.p. map
$$\fP_\mu = (\id\otimes\mu)\circ\alpha : A\to \ell^\infty(\G)$$
is called the {\it Poisson transform} of $\mu$. In fact, any $\G$-equivariant u.c.p. map $\Psi: A\to \ell^\infty(\G)$ is a Poisson transform of $\mu$,  where $\mu = \varepsilon_\G\circ\Psi$.
Let $A$ be a $\G$-\cst-algebra with a corresponding action $\alpha:\G\act A$. The {\textit {co-kernel}} of the action $\alpha$ is the von Neumann algebra
\[
N_\alpha=\{ \fP_\mu(a) : a\in A, \mu\in S(A)\}{''}.
\]
It is known that  $N_\alpha$ is a Baaj-Vaes subalgebra of $\ell^\infty(\G)$ \cite{MR4442841}*{Proposition 2.9}, so there exists a closed quantum subgroup $\hh{\G}_\alpha\leq \hh\G$ such that $N_\alpha = \ell^\infty(\G_\alpha)$. The action  $\alpha$ is called {\it faithful} if $\G_\alpha = \G$.

\begin{remark}\label{ex} Let $\G$ be a discrete quantum group.
  Let $\pi : C^u(\hh\G)\to B(H_\pi)$ be a $\G$-representation and $U_\pi\in \ell^\infty(\G)\vtp \B(H_\pi)$ the corresponding unitary corepresentation. It is clear that $B(H_\pi)$ and the induced \cst-algebra $C^\pi(\hh\G):= \pi(C^u(\hh\G))$ are $\G$-\cst-algebras via the action $\ad_\pi(\cdot):= U_\pi^*(1\otimes a)U_\pi$.
\end{remark}

We expand the concept of amenability for a discrete quantum group to $\G$-operator subsystems of  $\ell^\infty(\G)$  with the action defined by the restricted comultiplication. 
%We extend the definition of amenability of a discrete quantum group to  $\G$-operator subsystems of $\ell^\infty(\G)$
% where the action is given by the restriction of comultiplication. 
%We introduce generalizations of  the definition of amenability of a discrete quantum group to that the action is given by the restriction of comultiplication.
Note that  a discrete quantum group $\G$ is {\it amenable} if $\ell^\infty(\G)$ admits a (left) invariant mean,  i.e. a $\G$-equivariant u.c.p. map $m : \ell^\infty(\G)\to \C$.
\begin{definition}
    Let $\G$ be a discrete quantum group  and $X$ be a  weak$^*$ closed operator subsystem of $\ell^\infty(\G)$. Let $\Delta_\G(X) \subset \ell^\infty(\G){\otimes}_{\fF} X$,  define a $\G$-action on $X$. Then $X$ is called
    \begin{itemize}
        \item  {\it relatively amenable} if there exists a $\G$-equivariant u.c.p. map from $\ell^\infty(\G)$ into $X$.
        \item {\it $\G$-injective} if there exists a $\G$-equivariant u.c.p. projection from $\ell^\infty(\G)$ onto $X$.
    \end{itemize}
\end{definition}
    The $\G$-injectivity of $X\subseteq \ell^\infty(\G)$ as defined above is equivalent to $X$ being injective in the category of $\G$-operator systems with $\G$-equivariant u.c.p. maps as morphisms. This equivalence holds because $\ell^\infty(\G)$ is $\G$-injective \cite{MR4442841}*{Proposition 4.14}. More precisely, $X$ is $\G$-injective iff the following holds: for $\G$-operator systems $Y$ and $Z$ with a $\G$-equivariant u.c.i. map $\iota : Y\to Z$, if there is a $\G$-equivariant u.c.p. map $\varphi : Y\to X$, then there exists a $\G$-equivariant u.c.p. map $\tilde{\varphi} : Z\to X$ such that $\tilde{\varphi}\circ \iota = \varphi$.

An alternative characterization of relative amenability for coideals, linked to the presence of an invariant mean type condition, is presented in \cite{AS21(1)}. We provide this description for the reader's convenience.
\begin{definition}
Let $E$ be a subspace of $\ell^1(\G)$. A state $m:\ell^\infty(\G)\to \C$ is called {\it $E$-invariant} if for all $f\in E$, $x\in\ell^\infty(\G)$, $f * m(x)= m((f\otimes\id)\Delta_\G(x)) =f(1) m(x)$.
\end{definition}
Note that given $f\in \ell^1(\G)$, we let $fP_\X, P_\X f\in\ell^1(\G)$ be the functionals defined by setting
$$fP_\X(x) = f(P_\X x), P_\X f(x) = f(x P_\X), x\in \ell^\infty(\G),$$
and we let $P_\X \ell^1(\G) = \{P_\X f : f\in\ell^1(\G)\}$ and $\ell^1(\G)P_\X = \{fP_\X : f\in\ell^1(\G)\}$.
\begin{theorem}\cite{AS21(1)}*{Theorem 3.4}\label{RAcoideal}
    Let $\G$ be a discrete quantum group and $\ell^\infty(\X)\subseteq \ell^\infty(\G)$ be a coideal with the associated group-like projection $P_\X$. The following are equivalent:
    \begin{enumerate}
        \item $\ell^\infty(\X)$ is relatively amenable;
        \item there exists a $P_\X \ell^1(\G)$-invariant state $m_\X : \ell^\infty(\G)\to \C$;
        \item there exists a $\ell^1(\G)P_\X$-invariant state $m_\X : \ell^\infty(\G)\to \C$.
    \end{enumerate}
\end{theorem}

%We recall the space  $M_{P_\X}$, introduced in \cite{AS21(1)} as follows
%\[
%M_{P_\X} := \{x\in \ell^\infty(\G) : (1\otimes P_\X)\Delta_\G(x)(\I\otimes P_\X) = x\otimes P_\X\}.
%\]

\subsection{Compact Quasi-subgroups of compact quantum groups}
As previously stated, the idempotent functional $\omega_{\hh\X} := h_{\hh\X}\circ q_{\hh\X}\in \fO(\hh\G)^*$, assigned to a quantum space $\hh\X$, may not be a positive functional. Following the terminology in \cite{MR4055973} and related references, we examine the case where $\omega_{\hh\X}$ is positive.
\begin{definition}
    Let $\fO(\hh\X\backslash\hh\G)$ be a coideal subalgebra of $\hh\G$. We call $\hh\X$ a {\it compact quasi-subgroup} of $\hh\G$ if $\omega_{\hh\X}$ is an idempotent state.
\end{definition}
In the case where $\hh\X$ is a compact quasi-subgroup, the map
$$E_{\omega_{\hh\X}} := (\omega_{\hh\X}\otimes\id)\circ\Delta_{\hh\G} : \fO(\hh\G)\to \fO(\hh\X\backslash\hh\G)$$
is a $\hh\G$-equivariant projection onto $\fO(\hh\X\backslash\hh\G)$, \cite{MR3863479}*{Lemma 4.3}. It extends to a $\hh\G$-equivariant u.c.p. projection from $C(\hh\G)$ onto $C(\hh\X\backslash\hh\G)$, a $\hh\G$-equivariant u.c.p. projection from $C^u(\hh\G)$ onto $C^u(\hh\X\backslash\hh\G)$, and a normal $\hh\G$-equivariant u.c.p. projection from $L^\infty(\hh\G)$ onto $L^\infty(\hh\X\backslash\hh\G)$. Moreover, the corresponding group-like projection satisfies $P_\X = (\omega_{\hh\X}\otimes\id)(\Ww^{\hh\G})$ and $P_\X= R_\G(P_\X) = Q_\X$. It is worth noting that $P_\X$ is $\tau^\G$-invariant iff $\hh\X$ is a compact quasi-subgroup (see \cite{MR3835453}*{Theorem 3.1 and Theorem 4.3}). %Furthermore, $\tau_\G$-invariant group-like projections in $\ell^\infty(\G)$ are in one-to-one correspondence with idempotent states in $C^u(\hh\G)^*$ (see [\cite{MR3835453}, Theorem 4.3 and Theorem 3.1]).

Given a compact quasi-subgroup $\hh\X$, we let $(L^2(\hh\X), \Gamma_{\X}, \eta_{\hh\X}, \Omega_{\hh\X})$ be the GNS construction of $\omega_{\hh\X}$,
\begin{equation}\label{Gamma}
\Gamma_\X:\fO(\hh\G)\to B(L^2(\hh\X)), \qquad \omega_{\hh\X}(\cdot) = \langle \Gamma_\X(\cdot) \Omega_{\hh\X},\Omega_{\hh\X}\rangle.
\end{equation}
The representation $\Gamma_{\X}$ is a $\G$-representation so its associated \cst-algebra $C^{\Gamma_{\X}}(\hh\G)$ is a $\G$-\cst-algebra (Remark \ref{ex}). In particular, if $\hh\h$ is a compact quantum subgroup of $\hh\G$ via the morphism $\Pi:\fO(\hh\G)\to \fO(\hh\h)$, then ${\Gamma_{\h}}=\lambda_{\h}\circ\Pi$.

Let $\hh\X$ be a compact quasi-subgroup of $\hh\G$. It turns out there is a decomposition $\prod_{\alpha\in I_\X} B(H_\alpha) \cong \fO(\hh\X)^*\subseteq \prod_{u\in \Irr(\hh\G)}B(H_u)$, where $\dim(H_\alpha) < \infty$ for all $\alpha$ (see, for example, \cite{MR3863479}). We set $c_{00}(\X) := \oplus_{\alpha\in I_\X}B(H_\alpha)\subseteq \ell^\infty(\G)$. It turns out that $c_{00}(\X)'' =\ell^\infty(\X)$. Indeed, we have that $(1\otimes P_\X)\Delta_\G(x) = x\otimes P_\X$ for all $x\in c_{00}(\X)$ (see, for example, the proof of \cite{DCDT21}*{Lemma 1.3}) and hence $P_\X x = \varepsilon_\G(x)P_\X$ for all $x\in c_{00}(\X)$ which implies $P_\X\in c_{00}(\X)$ because then $P_\X\in \prod_{\alpha\in I_\X} B(H_\alpha)$ is a projection onto a one-dimensional Hilbert space (see, also, \cite{DCDT21}*{Lemma 1.6 and Definition 1.5}). Hence, $c_{00}(\X)''$ is a von Neumann algebraic coideal of $\G$ such that $P_\X\in c_{00}(\X)'' \subseteq \ell^\infty(\X)$. Then, if we let $Q\in c_{00}(\X)''$ be the associated group-like projection by Remark \ref{Coideals of Quantum Groups}, we must have $P_\X = Q$ by minimality and hence $c_{00}(\X)'' = \ell^\infty(\X)$ \cite{MR3835453}.

As a result of the discussion in the above paragraph and \cite{DCDT21}*{Theorem 1.12}, there exists a faithful normal positive functional $\psi_\X : c_{00}(\X)\to \C$ such that $f*\psi_\X = f(1)\psi_\X$ for all $f\in\ell^1(\G)$ (this sort of result was actually known before \cite{DCDT21} (see the proof of \cite{MR3863479}*{Proposition 4.5})).

Let $L^2(\ell^\infty(\X), \psi_\X)$ be the GNS Hilbert space coming from $\psi_\X$. It turns out that there is a unitary isomorphism $L^2(\hh\X)\cong L^2(\ell^\infty(\X), \psi_\X):=\ell^2(\X)$ (cf. \cite{MR3863479}) that provides an analog of Plancherel's theorem for coideals. It gives us a nice duality theory for coideals between $\G$ and $\hh\G$. This duality allows us to take both embeddings $C^{\Gamma_{\X}}(\hh\G), \ell^\infty(\X)\subseteq B(\ell^2(\X))$ to be injective $*$-homomorphisms. Furthermore, the action $\G\act B(\ell^2(\X))$ restricts to the canonical actions $\G\act \ell^\infty(\X), C^{\Gamma_{\X}}(\hh\G)$. More precisely, we have
$ad_{{\Gamma_{\X}}}(\ell^\infty(\X))\subseteq \ell^\infty(\G)\otimes \ell^\infty(\X)$ and,
as mentioned in Remark \ref{ex}, $ad_{\Gamma_\X}(C^{\Gamma_\X}(\hh\G)) \subseteq \ell^\infty(\G)\otimes C^{\Gamma_\X}(\hh\G)$. This last claim was shown in \cite{ASV23}*{Section 4.1} (for a more general class of coideals). This claim was known by experts for compact quasi-subgroups before \cite{ASV23}, but we have no earlier reference. This analog of Plancherel's theorem and its consequences above will be essential.

\begin{definition} Let $\G$ be a discrete quantum group. Let $\pi:\fO(\hh\G)\to C^\pi(\hh\G)$ and $\sigma:\fO(\hh\G)\to C^\sigma(\hh\G)$ be $\G$-representations. We say $\pi$ is weakly contained in $\sigma$, written $\pi\prec \sigma$, if the map $\sigma(x)\to \pi(x)$ extends to a surjective $*$-homomorphism $C^\sigma(\hh\G)\to C^\pi(\hh\G)$, or equivalently if  $||\pi(\cdot)||\leq ||\sigma(\cdot)||$.
\end{definition}

It is known that a compact quantum group $\hh\G$ is {\it coamenable} if the counit $\varepsilon_{\hh\G} : \fO(\hh\G)\to \C$ extends continuously to $C(\hh\G)$. In other words, $\hh\G$ is coamenable if and only if $\varepsilon_{\hh\G}\prec \lambda_\G$.

A  notion of coamenability for  coideals  of quotient type (von Neumann algebraic) was introduced in \cite{MR4442841} which generalizes coamenability of a compact quantum group. In \cite{AS21(1)}, this definition is extended naturally to arbitrary coideals.
\begin{definition}
    Let $\G$ be a discrete quantum group and $\fO(\hh\X\backslash\hh\G)$ be a coideal subalgebra of $\hh\G$. We say $\hh\X\backslash\hh\G$ is  coamenable  if the restriction of $\varepsilon_{\hh\G}|_{\fO(\hh\X\backslash\hh\G)}$ extends continuously to a state on $C(\hh\X\backslash\hh\G)$.
\end{definition}
For a closed quantum subgroup $\hh\h$ of compact quantum group $\hh\G$,  coamenability of a coideal of quotient type $\hh\h\backslash\hh\G$ is equivalent with relative amenability of $\ell^\infty(\h)$, \cite{MR4442841}*{Theorem 3.11} .

This statement was generalized to compact quasi-subgroups in \cite{AS21(1)}*{Theorem 4.10} but it is phrased slightly differently. To elaborate, it was proved that for a compact quasi-subgroup $\hh\X$,  $\hh\X\backslash\hh\G$ is coamenable  iff $\omega_{\hh\X}$ factorizes through the left regular representation of $\G$. Since $\Gamma_\X$ is the GNS representation of $\omega_{\hh\X}$, this is clearly equivalent to having $\Gamma_\X\prec\lambda_\G$.

 \begin{remark}\label{Compact Quasisubgroups of Quantum Groups}
     Let us conclude this subsection by summarizing something we want to emphasize about compact quasi-subgroups of compact quantum groups. There is a bijection between the following sets of objects:
 \begin{itemize}
    \item compact quasi-subgroups $\hh\X$ of $\hh\G$;
     \item idempotent states $\omega_{\hh\X}\in C^u(\hh\G)^*$;
     \item group-like projections $P_\X\in \ell^\infty(\G)$ that are $\tau^\G$-invariant.
 \end{itemize}
    In particular, when we invoke any one of the above objects, the existence of each of the other objects is automatically implied. We will frequently use this fact.
 \end{remark}

   \section{(Co)amenability Properties of Coideals}\label{sec3}
Let $\ell^\infty(\X)$ be a coideal of a discrete quantum group $\G$. Recall that the space $M_{P_\X}$ is a weak$^*$-closed operator  subsystem of $\ell^\infty(\G)$ introduced in \cite{AS21(1)},
\begin{align}\label{MPh1}
  \ell^\infty(\X)\subseteq M_{P_{\X}} := \{x\in \ell^\infty(\G) : (1\otimes P_{\X})\Delta_\G(x)(1\otimes P_{\X}) = x\otimes P_{\X}\}\subseteq \ell^\infty(\G)
\end{align}
Note that $M_{P_\X}$ is a (left) $\G$-operator system with the  restriction of the comultiplication. Since for arbitrary coideals $\ell^\infty(\X)$ of $\G$ the corresponding group-like projection $P_\X\in \ell^\infty(\G)$ is not necessarily central we do not know whether $\ell^\infty(\X)$ equals to $M_{P_\X}$.  

 In this section, we will establish a characterization of $\G$-injectivity of a coideal of a discrete quantum group $\G$ that is analogous to the invariant state type characterizations of relative amenability of $\ell^\infty(\X)$  mentioned in Section \ref{2.4}. As an application, we obtain that for a compact quasi-subgroup $\hh\X$, coamenability of $\hh\X\backslash\hh\G$ implies $\G$-injectivity of $\ell^\infty(\X)$, and this yields
$\ell^\infty(\X) = M_{P_\X}$. We also are able to prove that $\G$-injectivity of $\ell^\infty(\X)$ yields nuclearity  of $C(\hh\X\backslash\hh\G)$.

We will first make some general observations about $M_{P_\X}$.
Let $\G$ be a discrete quantum group and let $\ell^\infty(\X)$ be a coideal of $\G$ with associated group-like projection $P_\X$. %Recall that
%\begin{align}\label{MPh1}
 % \ell^\infty(\X)\subseteq M_{P_{\X}} := \{x\in \ell^\infty(\G) : (1\otimes P_{\X})\Delta_\G(x)(1\otimes P_{\X}) = x\otimes P_{\X}\}\subseteq \ell^\infty(\G)
%\end{align}
Note that if $\ell^\infty(\X)$ is quotient type then $P_{\X}$ is central which implies $M_{P_{\X}} = \ell^\infty(\X)$ \cite{MR3552528}*{Theorem 4.3}.

A special case of the following theorem can be found in \cite{MR4055973}. Namely, if $L^\infty(\hh\X\backslash\hh\G)$ is finite-dimensional then $\ell^\infty(\X) = M_{P_{\X}}$. Let's outline the rationale briefly. It is proven in \cite[Corollary 4.2]{MR4055973} that all idempotent states of a discrete quantum group $\G$ are normal and they correspond one-to-one with normal idempotent states of $\hh\G$, as stated in \cite[Theorem 4.12]{MR4055973}. By employing \cite{MR4055973}*{Theorem 4.16},  we conclude  that  idempotent states in $\ell^1(\G)$ are in bijection with  finite-dimensional coideals of $\hh\G$.
Assume $L^\infty(\hh\X\backslash\hh\G)$ is finite-dimensional and let $m_{\X}\in \ell^1(\G)$ be the corresponding idempotent state. 
The kernel of $m_\X$, $J_{m_\X}= \{x\in \ell^\infty(\G) \st m_\X(x^*x)=0  \}$ is a weak$^*$ closed left ideal of $\ell^\infty(\G)$. If we let $P_{\X}^{\perp}$  the complement of the projection $P_\X$, then we have $J_{m_\X} =\ell^\infty(\G)P_{\X}^{\perp}$ \cite[Lemma 4.4]{MR4055973}. %and we have, 
%\[
%J_{m_\X} = \{x\in \ell^\infty(\G) \st m_\X(x^*x)=0  \}=\ell^\infty(\G)P_{\X}^{\perp},
%\]
%where $P_{\X}^{\perp}$ is the complement of $P_\X$, \cite[Lemma 4.4]{MR4055973}.
Therefore, we have 
$$m_{\X}(P_{\X}\cdot P_{\X})=m_{\X}(P_{\X}\cdot) = m_{\X}(\cdot P_{\X}) = m_\X(\cdot),$$
Let $E := (\id\otimes m_{\X})\circ\Delta_\G: \ell^\infty(\G)\to \ell^\infty(\X)$ be a (normal) $\G$-equivariant u.c.p. projection onto $\ell^\infty(\X)$. Then, it is easy to check that for every $x\in M_{P_{\X}}$, $E(x)=x$.
%It turns out $P_{\X}$ is the complement of the projection that implements the weak$^*$ closed left ideal $\{x\in \ell^\infty(\G) : m_{\X}(x^*x) = 0\}$. From this, we have
%and it is then easy to check that $E(x) = x$ for every $x\in M_{P_{\X}}$.

Recall that given $f\in \ell^1(\G)$, we let $fP_\X, P_\X f\in\ell^1(\G)$ be the functionals defined by setting
$$fP_\X(x) = f(P_\X x), P_\X f(x) = f(x P_\X), x\in \ell^\infty(\G)$$
and we let $P_\X \ell^1(\G) = \{P_\X f : f\in\ell^1(\G)\}$ and $\ell^1(\G)P_\X = \{fP_\X : f\in\ell^1(\G)\}$.

\begin{theorem}\label{Amenability Theorem for Mp}
    Let $\G$ be a discrete quantum group and $\fO(\hh\X\backslash\hh\G)$ be a coideal subalgebra of $\hh\G$. We have that $\ell^\infty(\X)$ is $\G$-injective if and only if there exists a $\ell^1(\G)P_\X$-invariant state $m : \ell^\infty(\G)\to\C$ such that
    \begin{align}
        m(P_{\X}\cdot) = m(\cdot P_{\X}) = m(\cdot). \label{Amenable State Condition}
    \end{align}
    Moreover, if $\ell^\infty(\X)$ is $\G$-injective then $\ell^\infty(\X) = M_{P_{\X}}$.
\end{theorem}

\begin{proof}
    Let $E : \ell^\infty(\G)\to \ell^\infty(\X)$ be a $\G$-equivariant u.c.p. projection onto $\ell^\infty(\X)$. It follows that $m = \varepsilon_\G\circ E$ is a $\ell^1(\G)P_{\X}$-invariant state \cite{AS21(1)}*{Theorem 3.3}, and $m(P_{\X}) = 1$. To justify the latter, simply note that
    $$m(P_{\X}) = \varepsilon_\G(E(P_{\X})) = \varepsilon_\G(P_{\X}) = 1.$$
    From this we make the observation
    $$\overline{m(P_{\X})}m(P_{\X}) = 1 = m(P^*_{\X}P_{\X}) = m(P_{\X}P^*_{\X}).$$
    Hence, $P_{\X}$ lies in the multiplicative domain of the u.c.p. map $m : \ell^\infty(\G)\to \C$ and we deduce that
    $$m(P_{\X}\cdot) = m(\cdot P_{\X}) = m(\cdot).$$
    Conversely, assume the existence of a $\ell^1(\G)P_\X$-invariant state $m : \ell^\infty(\G)\to\C$ such that $m(P_\X\cdot) = m(\cdot) = m(\cdot P_\X)$. Define $E$ by setting
    $$E(x) = (\id\otimes m)\Delta_\G(x),\quad x\in \ell^\infty(\G)$$
    It is easy to see that $E$ is a $\G$-equivariant u.c.p. map. Since $m$ is $\ell^1(\G)P_\X$-invariant, $E(\ell^\infty(\G))\subseteq \ell^\infty(\X)$ \cite{AS21(1)}*{Theorem 3.3}. Let $x\in M_{P_{\X}}$,
    \begin{align*}
        E(x)&= (\id\otimes m)\Delta_\G(x) = (\id\otimes m)(1\otimes P_{\X})\Delta_\G(x)(1\otimes P_{\X}) =(\id\otimes m)(x\otimes P_\X)= x.
    \end{align*}
   Therefore, $E$ is a projection onto $M_{P_\X}$, and since $\ell^\infty(\X)\subseteq M_{P_\X}$ we have, $E(\ell^\infty(\G)) = M_{P_{\X}} = \ell^\infty(\X)$.
\end{proof}

\begin{theorem}\label{Coamenability Implies Amenability: CQSs}
    Let $\G$ be a discrete quantum group  and $\hh\X$ be a compact quasi-subgroup of $\hh\G$. If $\hh\X\backslash\hh\G$ is  coamenable  then $\ell^\infty(\X)$ is $\G$-injective.
\end{theorem}
\begin{proof}
It is known that coamenability of $\hh\X\backslash\hh\G$ is equivalent with
 $\omega_{\hh\X}\in C(\hh\G)^*$ by \cite{AS21(1)}*{Theorem 4.10}.  Let  $(\xi_\alpha)\subseteq P_\X \ell^2(\G)$   be a net of unit vectors such that $\omega_{\xi_\alpha} \xrightarrow{weak^*} \omega_{\hh\X}$, like in \cite{AS21(1)}*{Corollary 4.12}.
    Let $m : B(\ell^2(\G))\to \C$ be a weak$^*$ cluster point of the net $(\omega_{\xi_\alpha})$ (which necessarily extends $\omega_{\hh\X}$). Recall that $P_\X = (\id\otimes \omega_{\hh\X})(\ww_\G)$. Take $f\in \ell^1_F(\G)$ and $x\in \ell^\infty(\G)$,
    \begin{align*}
        (f P_{\X}\otimes m)\Delta_\G(x)
    &= (f\otimes m)\left((P_{\X}\otimes 1)\ww^*_\G(1\otimes x)\ww_\G\right)
        \\
        &= (f\otimes \id)\left(\left(\id\otimes \omega_{\hh\X}\right)\left(\left(P_{\X}\otimes 1\right)\ww^*_\G\right)\left(\id\otimes m\right)\left(\left(1\otimes x\right)\ww_\G\right)\right)
        \\
        &= (f\otimes m)\left((P_\X \otimes x)\ww_\G\right)
        \\
        &=m(x) (f\otimes \omega_{\hh\X})( (P_\X\otimes 1)\ww_\G )
        = f(P_\X)m(x),
    \end{align*}
where in the second and fourth equality we respectively used the facts that  $\fO(\hh\X\backslash \hh\G)$ and $\fO(\hh\G / \hh\X)$ are subsets of the multiplicative domain of $m$ (which is true because $m|_{\fO(\hh\X\backslash\hh\G)} = \omega_{\hh\X}|_{\fO(\hh\X\backslash\hh\G)} = \varepsilon_{\hh\G}|_{\fO(\hh\X\backslash\hh\G)}$ and $m|_{\fO(\hh\G / \hh\X)} = \omega_{\hh\X}|_{\fO(\hh\G/\hh\X)} = \varepsilon_{\hh\G}|_{\fO(\hh\G / \hh\X)}$).    Furthermore, $\omega_{\xi_\alpha}(P_{\X}\cdot) = \omega_{\xi_\alpha}(\cdot P_{\X}) = \omega_{\xi_\alpha}(\cdot)$ for all $\alpha$ hence the same is true of $m|_{\ell^\infty(\G)}$ as well. Now apply Theorem \ref{Amenability Theorem for Mp}.
\end{proof}

We can use the statement of \cite{M18}*{Theorem 4.7} to prove the next corollary. Unfortunately, the proof of \cite{M18}*{Theorem 4.7} does not work beyond the unimodular case. It is assumed that there exists a $\G$-equivariant u.c.p. projection $E_0 : B(\ell^2(\G))\to \ell^\infty(\G)$  such that $E_0|_{L^\infty(\hh\G)} = h_{\hh\G}$ \cite{M18}*{(2.3)}. This latter condition, however, implies $h_{\hh\G}$ is $\G$-invariant, which implies $\G$ is unimodular thanks to \cite{MR4442841}*{Lemma 5.2}.
We do not see how to alter the proof of \cite{M18}*{Theorem 4.7} so that the condition $E_0|_{L^\infty(\hh\G)} = h_{\hh\G}$ is not required. Therefore, the following appears to be new.
\begin{corollary}\label{Amenable Mp Corollary}
    Let $\G$ be an amenable discrete quantum group. Then, for every compact quasi-subgroup $\hh\X$,  $\ell^\infty(\X)$ is $\G$-injective. Therefore, $\ell^\infty(\X) = M_{P_{\X}}$ holds for every compact quasi-subgroup $\hh\X$.
\end{corollary}
\begin{proof}
    It follows that $\hh\G$ is coamenable by \cite{MR2276175}*{Theorem 3.8}. It is clear that, for every compact quasi-subgroup $\hh\X$ of $\hh\G$, $\hh\X\backslash\hh\G$ is  coamenable  and then we apply Theorem \ref{Coamenability Implies Amenability: CQSs}.
\end{proof}
Recall that $C^u(\hh\X\backslash\hh\G)$ is defined to be the closure of the embedding of $\fO(\hh\X\backslash\hh\G)$ inside of $C^u(\hh\X\backslash\hh\G)$. Here, we identify $C^u(\hh\G)$ concretely with $C^{\varpi_\G}(\hh\G)$, where $\varpi_\G$ is the universal $\G$-representation.

Consider the universal representation of $C^u(\hh\X\backslash\hh\G)$:
$$\varpi_{\hh\X\backslash\hh\G} := \bigoplus_{\mu\in S(C^u(\hh\X\backslash\hh\G))} \pi_\mu : C^u(\hh\X\backslash\hh\G) \to B(H_{\varpi_{\hh\X\backslash\hh\G}})$$
where $\pi_\mu$ is the GNS representation obtained from $\mu$. Write $C^{\varpi}(\hh\X\backslash\hh\G) = \varpi_{\hh\X\backslash\hh\G}(C^u(\hh\X\backslash\hh\G))$. Note that $\varpi_{\hh\X\backslash\hh\G}$ is faithful and $C^{\varpi}(\hh\X\backslash\hh\G)'' \cong C^u(\hh\X\backslash\hh\G)^{**}$ by universality.
\begin{remark}
    Apriori $C^u(\hh\X\backslash\hh\G)$ is not the maximal \cst-cover of $\fO(\hh\X\backslash\hh\G)$. It turns out that it is in the case where $\hh\X$ is a compact quasi-subgroup (see \cite{AS22(phd)}*{Proposition 5.2.3}). On the other hand, it appears unknown whether or not $C^u(\hh\X\backslash\hh\G)$ is the maximal \cst-cover of $\fO(\hh\X\backslash\hh\G)$ for arbitrary coideals.
\end{remark}

Recall that the $\hh\G$-equivariant projection $E_{\omega_{\hh\X}} : \fO(\hh\G) \to \fO(\hh\X\backslash\hh\G)$ admits a universal version $E_{\omega_{\hh\X}}^u : C^u(\hh\G)\to C^u(\hh\X\backslash\hh\G)$.

As our last application of Theorem~\ref{Amenability Theorem for Mp}, we obtain the following.
\begin{theorem}\label{GInjective Implies Nuclear}
    Let $\G$ be a discrete quantum group and $\hh\X$ be a compact quasi-subgroup of $\hh\G$. If $\ell^\infty(\X)$ is $\G$-injective then $C(\hh\X\backslash\hh\G)$ is nuclear.
\end{theorem}
\begin{proof}
    Let $\Psi : \ell^\infty(\G)\to \ell^\infty(\X)$ be the $\G$-equivariant conditional expectation given by assumption. Let $\tilde{m}=\varepsilon_\G \circ \Psi : \ell^\infty(\G)\to\C$ be the associated $\ell^1(\G)P_\X$-invariant state which also satisfies $\tilde{m}(P_\X\cdot) = \tilde{m}(\cdot P_\X) = \tilde{m}(\cdot)$ thanks to Theorem~\ref{Amenability Theorem for Mp}. Set $m = \tilde{m}\circ R_\G$. It is straightforward to check that $\Phi := (m\otimes\id)\circ \Delta_\G$ is a conditional expectation $\ell^\infty(\G) \to R_\G(\ell^\infty(\X))$ onto $R_\G(\ell^\infty(\X))$ satisfying $(\Phi\otimes\id)\circ\Delta_\G = \Delta_\G\circ\Phi$.

    We will prove that $C^{\varpi}(\hh\X\backslash\hh\G)'$ is injective. Then $C^{\varpi}(\hh\X\backslash\hh\G)'' \cong C^u(\hh\X\backslash\hh\G)^{**}$ is injective which is equivalent to the nuclearity of $C^u(\hh\X\backslash\hh\G)$. Moreover, the quotient of a nuclear \cst-algebra is nuclear, so we deduce that $C(\hh\X\backslash\hh\G)$ is nuclear as desired.

    First, we explicitly denote the $*$-isomorphism
    $\Lambda_\X : C^u(\hh\X\backslash\hh\G)\to C^{\varpi}(\hh\X\backslash\hh\G)$, $\varpi_\G(a)\mapsto \varpi_{\hh\X\backslash\hh\G}(a)$.  Recall that
    $\wW_\G(P_\X\otimes 1) = (\id\otimes E^u_{\omega_{\hh\X}})(\wW_\G)$ by \cite{integrable}*{Theorem 7.4} and additional clarifications at the beginning of Section  $3$ in \cite{MR3835453}. Let us define $\wW_\X := (\id\otimes \Lambda_{\X}\circ E^u_{\omega_{\hh\X}})\wW_\G$ and set $$\mathrm{ad}_{\varpi_{\hh\X\backslash\hh\G}} : B(H_{\varpi_{\hh\X\backslash\hh\G}})\to \ell^\infty(\G)\overline{\otimes}B(H_{\varpi_{\hh\X\backslash\hh\G}}),\qquad
   \mathrm{ad}_{\varpi_{\hh\X\backslash\hh\G}}(\cdot) = \wW_\X^*(1\otimes (\cdot))\wW_\X.$$
    
    It is straightforward to observe that $(\id\otimes \mathrm{ad}_{\varpi_{\hh\X\backslash\hh\G}})\circ \mathrm{ad}_{\varpi_{\hh\X\backslash\hh\G}} = (\Delta_\G\otimes\id)\circ\mathrm{ad}_{\varpi_{\hh\X\backslash\hh\G}}$. We can define the map $E : B(H_{\varpi_{\hh\X\backslash\hh\G}})\to B(H_{\varpi_{\hh\X\backslash\hh\G}})$ by setting $E(T) =  (m\otimes \id)\circ \mathrm{ad}_{\varpi_{\hh\X\backslash\hh\G}}(T)$. We claim that $E : B(H_{\varpi_{\hh\X\backslash\hh\G}})\to C^{\varpi}(\hh\X\backslash\hh\G)'$ is a conditional expectation onto $C^{\varpi}(\hh\X\backslash\hh\G)' = \varpi_{\hh\X\backslash\hh\G}(C^u(\hh\X\backslash\hh\G))'$. It is clear that $E$ is a u.c.p. map. Observe that
    \begin{align*}
        \mathrm{ad}_{\varpi_{\hh\X\backslash\hh\G}}(E(T)) &= (P_\X\otimes 1)\mathrm{ad}_{\varpi_{\hh\X\backslash\hh\G}}(E(T)) ~~ \text{(since $\wW_\X(P_\X\otimes 1) = \wW_\X$)}
        \\
        &= (m\otimes\id\otimes \id)\left((1\otimes P_\X\otimes 1)\left(\id\otimes \mathrm{ad}_{\varpi_{\hh\X\backslash\hh\G}})\circ\mathrm{ad}_{\varpi_{\hh\X\backslash\hh\G}}(T)\right)\right)
        \\
        &= (m\otimes\id\otimes \id)\left((1\otimes P_\X\otimes 1)\left((\Delta_\G\otimes\id)\circ\mathrm{ad}_{\varpi_{\hh\X\backslash\hh\G}}(T)\right)\right)
        \\
        &= (P_\X\otimes 1)\left(( \Phi\otimes\id)\circ\mathrm{ad}_{\varpi_{\hh\X\backslash\hh\G}}(T)\right)
        \\
        &=  P_\X\otimes (m\otimes \id)\circ \mathrm{ad}_{\varpi_{\hh\X\backslash\hh\G}}(T) ~~ \text{(since {$P_\X\Phi(x) = \varepsilon_\G(\Phi(x))P_\X = m(x)P_\X$})}
        \\
        &= P_\X\otimes E(T)
    \end{align*}
   Where the fact that  $P_\X\Phi(x) = \varepsilon_\G(\Phi(x))P_\X$ can be found in \cite{AS21(1)}*{Lemma 3.1}
    and
    \begin{align}\label{Expn Equ}
        \mathrm{ad}_{\varpi_{\hh\X\backslash\hh\G}}(E(T)) = P_\X\otimes E(T)
        \iff (1\otimes E(T))W_\X = W_\X(1\otimes E(T)).
    \end{align}
    Then, from the combination of Equation~\eqref{Expn Equ} with the equality
    $$C^u(\hh\X\backslash\hh\G) = \overline{ \{ (f\otimes\id)\circ (\wW_\G(P_\X\otimes 1)) : f\in \ell^1(\G) \} }^{||\cdot||}$$
    we conclude that $E(T)\in \varpi_{\hh\X\backslash\hh\G}(C^u(\hh\X\backslash\hh\G))'$. On the other hand, given $T\in \varpi_{\hh\X\backslash\hh\G}(C^u(\hh\X\backslash\hh\G))'$, $(1\otimes T)\wW_\X = \wW_\X(1\otimes T)$ and so
    \begin{align*}
      E(T)  = (m\otimes\id)\circ \mathrm{ad}_{\varpi_{\hh\X\backslash\hh\G}}(T) &=  (m\otimes\id)\left(\wW_\X^* \wW_\X(1\otimes T)\right) =  T.
    \end{align*}
    By applying $m\otimes \id$ to the equation $\mathrm{ad}_{\varpi_{\hh\X\backslash\hh\G}}  \left( E(T)\right) = P_\X \otimes E(T)$ we deduce idempotency of $E$.
\end{proof}

\section{Furstenberg-Hamana Boundary for a representation}\label{sec4}
\subsection{Constructing Furstenberg-Hamana boundaries for $\mathbf{\G}$-representations}

In the classical case where $G$ is a discrete group and  $\pi: G \to B(H_\pi)$ is a unitary representation, Bearden and Kalantar  \cite{MR4276323} studied a relative $G$-injective envelope of $\C\subseteq B(H_\pi)$ based on the ``non-relative'' constructions due to Hamana \cites{MR0519044, MR0566081, MR1189167}. This yields  a notion of a (topological) boundary for an arbitrary $G$-representation $\pi$, which is called the Furstenberg-Hamana boundary   (FH-boundary) of the representation $\pi$.
In this section,  we are interested in such objects for discrete quantum groups. In fact, all the methods in \cite{MR4276323} pass verbatim from the classical situation to the quantum setting.

For a discrete quantum group $\G$, a $\G$-boundary is a $\G$-\cst-algebra $A$ such that every $\G$-equivariant u.c.p. map from $A$ into $\ell^\infty(\G)$ is completely isometric, \cite{MR4442841}*{Definition 4.1}.
   The authors in \cite{MR4442841} adapted Hamana's injective envelope construction for discrete quantum groups and proved the existence of a unique universal $\G$-boundary, known as the non-commutative Furstenberg boundary, and denoted it by $C(\partial_F \G)$.
   For a $\G$-representation $\pi$, we use Hamana's method \cites{MR0566081,Ham85} and the procedure used in  \cite{MR4276323}
 in order to construct a Furstenberg-Hamana  boundary for  a representation of a discrete quantum group.

Let $\G$ be a discrete quantum group and $\pi:\fO(\hh\G)\to B(H_\pi)$ be a $\G$-representation. We pointed out that $B(H_\pi)$ is a $\G$-\cst-algebra in Remark \ref{ex}.
We denote  the set of all $\G$-equivariant u.c.p. maps $B(H_\pi)\to B(H_\pi)$ by $\mathcal{G}_\pi$. We equip $\mathcal{G}_\pi$ with the following partial order
\[
\Phi\leq\Psi ~\text{if} ~ ||\Phi(x)||\leq ||\Psi(x)|| ~ \text{for all} ~ x\in B(H_\pi).
\]
The set $\mathcal{G}_\pi$ contains a minimal element $\Phi_0$ \cite{MR4276323}*{Proposition 3.3}. We denote the $\Image(\Phi_0)$ by $\B_\pi$ which is a \cst-algebra via the Choi-Effros product. Note that   the image of a minimal element of $\mathcal{G}_\pi$ is unique up to isomorphism in the category of $\G$-\cst-algebras with $\G$-equivariant u.c.p. maps as morphisms \cite{MR4276323}*{Proposition 3.5}.
Moreover, $\B_\pi$ is a $\G$-\cst-algebra with a (left) action $\alpha_\pi$ given by $\alpha_\pi=\mathrm{ad}_\pi|_{\B_\pi}$, \cite{MR4276323}*{Theorem 4.9}.
\begin{proposition}\label{Relative G-properties} Let $\G$ be a discrete quantum group. For every $\G$-representation $\pi$, the \cst-algebra $\B_\pi$ has the following properties:

    \begin{enumerate}
         \item { $\pi$-rigidity:} the identity is the unique $\G$-equivariant u.c.p. map on $\B_\pi$;
        \item { $\pi$-essentiality:} Every $\G$-equivariant u.c.p. map $\B_\pi\to B(H_\pi)$ is completely isometric;
        \item { $\pi$-injectivity:} If $X\subseteq Y\subseteq B(H_\pi)$ are $\G$-invariant subspaces of $B(H_\pi)$ and there exists a $\G$-equivariant u.c.p. map $\Psi : B(H_\pi)\to\ B(H_\pi)$ such that $\Psi(X) \subseteq \B_\pi$ then there exists a $\G$-equivariant u.c.p. map $\tilde{\Psi} : B(H_\pi)\to  B(H_\pi)$ such that $\tilde{\Psi}(Y) \subseteq \B_\pi$ and $\tilde{\Psi}|_X = \Psi|_X$.
    \end{enumerate}
\end{proposition}
\begin{proof}
   The proof is similar to the proof of \cite{MR4276323}*{Proposition 3.4}.
\end{proof}

\begin{remark}\label{F Boundary}
The Furstenberg-Hamana boundary of the left regular representation $
\lambda_\G:\fO(\hh\G)\to B(\ell^2(\G))$
is the non-commutative Furstenberg boundary of $\G$, i.e. $\B_{\lambda_\G}= C(\partial_F\G)$. Therefore, $C(\partial_F\G)$ enjoys the following properties \cite{MR4442841}*{Proposition 4.10, Proposition 4.13}:
\begin{itemize}
 \item  $\G$-rigidity: the identity map is the unique $\G$-equivariant u.c.p. map on $C(\partial_F\G)$;
    \item $\G$-essentiality: Every $\G$-equivariant u.c.p. map $C(\partial_F\G)\to A$ into any $\G$-\cst-algebra $A$ is completely isometric.
\item $\G$-injectivity: for any $\G$-\cst-algebras $A$ and $B$ equipped with $\G$-equivariant u.c.p. maps
$\Psi:A\to C(\partial_F\G)$ and $\iota: A\to B$, with $\iota$ completely isometric, there exists a $\G$-equivariant u.c.p.
map $\tilde{\Psi}:B\to C(\partial_F\G)$  such that $\Psi=\tilde{\Psi}\circ\iota.$
\end{itemize}
\end{remark}
We also have the following categorical property for $FH$-boundaries of $\G$-representations.
\begin{proposition}\label{Categorical Property of FH Boundaries}
    Let  $\pi$ and $\sigma$ be $\G$-representations of a discrete quantum group $\G$. There exists a $\G$-equivariant u.c.p. map $B(H_\pi)\to B(H_\sigma)$ if and only if there exists a $\G$-equivariant u.c.p. map $\B_\pi\to \B_\sigma$. In particular, if $\sigma$ is  weakly contained in $\pi$ then there is a $\G$-equivariant u.c.p. map $\B_\pi\to \mathcal{B}_\sigma$.
\end{proposition}
\begin{proof}
    The proof is similar to the proof of \cite{MR4276323}*{Proposition 3.17}.
\end{proof}
Consequently, weak equivalence of $\G$-representations implies their $FH$-boundaries are isomorphic \cite{MR4276323}*{Corollary 3.18}. The converse does not hold in general. For instance, the FH-boundary of universal and trivial representations are trivial, $\B_{\varpi_\G} = \B_{\varepsilon_{\hh\G}} = \C$

In the case of a $\G$-representation induced by the Haar state of a closed quantum subgroup, the $FH$-boundary is the (n.c.) Furstenberg boundary of the quantum subgroup.
\begin{proposition}\label{Furstenberg Boundaries of Subgroups Are FH Boundaries}
    Let $\G$ be a discrete quantum group. Let $\hh\h$ be a closed quantum subgroup of $\hh\G$ associated with a morphism $\Pi:\fO(\hh\G)\to \fO(\hh\h)$. Consider the $\G$-representation
    $\Gamma_{\h}=\lambda_{\h}\circ \Pi : \fO(\hh\G)\to C(\hh\h)$.
    We have $\B_{\Gamma_\h} = C(\partial_{F}\h)$.
\end{proposition}
\begin{proof}
    Since $\Delta_\G|_{\ell^\infty(\h)} = \Delta_\h$, it is easy to show that a u.c.p. map $B(\ell^2(\h))\to B(\ell^2(\h))$ is $\G$-equivariant if and only if it is $\h$-equivariant. The claim then follows immediately from Remark \ref{F Boundary}.
\end{proof}

Recall that  the cokernel of an action $\alpha:\G\act A$ is a quantum group denoted by $\G_\alpha$. In particular, we will use the following notations for cokernels of some special actions:
\begin{itemize}
    \item For
    $\alpha_\pi:\G\act \B_\pi$ we write $\G_\pi := \G_{\alpha_{\pi}}$;
    \item For $\alpha_F:\G\act C(\partial_F \G)$ we write $\G_F := \G_{\alpha_F}$.
\end{itemize}
The von Neumann algebra $\ell^\infty(\G_F)$ is the unique minimal relatively amenable Baaj-Vaes subalgebra of $\ell^\infty(\G)$. This means that for every closed quantum subgroup $\hh\h\leq\hh\G$, if $\ell^\infty(\h)$ is a relatively amenable von Neumann subalgebra of $\ell^\infty(\G)$, then  $\hh\G_F\leq \hh\h$ (cf. \cite{MR4442841}*{Theorem 5.1}).

\subsection{Furstenberg-Hamana boundaries and relative amenability}
We now work towards the main result of this section. The result is a characterization of relative amenability of $\ell^\infty(\X)$ (in the context of compact quasi-subgroups) in terms of the existence of a $\G$-equivariant isomorphism $\B_{\Gamma_\X}\cong C(\partial_F\G)$.

First, we define lemmas that describe the relative amenability of coideals based on their interaction with the Furstenberg boundary of $\G$. These lemmas are crucial for proving our main result and may be of independent interest as they are applicable to any coideals.

Let $\G$ be a discrete quantum group. In what follows, given a $\G$-\cst-algebra $A$ and $\mu\in A^*$, we will let
$$\Fix(\mu) := \{f\in \ell^1(\G) : f*\mu = f(1)\mu\}.$$
\begin{lemma}\label{Stationarity Lemma}
    Let $A$ be a $\G$-\cst-algebra and $\mu\in A^*$.  For every $f\in \ell^1(\G)$, $f*\mathcal{P}_\mu(a)=f(1)\mathcal{P}_\mu(a)$ holds for every $a\in A$ if and only if
    $f\in\Fix(\mu)$.
\end{lemma}
\begin{proof}
    Assume  $f\in \Fix(\mu)$,  then for every $a\in A$,
    $$f*\mathcal{P}_\mu(a) = (\id\otimes f)\Delta_\G(\mathcal{P}_\mu(a)) =(\id\otimes f*\mu)\alpha(a) = f(1)\mathcal{P}_\mu(a).$$
    Conversely, if $f*\mathcal{P}_\mu(a) = f(1)\mathcal{P}_\mu(a)$, then
    
    \begin{equation*}
    f*\mu(a) = \varepsilon_\G*f*\mu(a) = (\varepsilon_\G * f)\Delta_\G(\fP_\mu(a))= \varepsilon_\G(f*\mathcal{P}_\mu(a)) = f(1)\mu(a).
    \qedhere
    \end{equation*}
\end{proof}

\begin{lemma}\label{Relative Amenability Lemma}
    Let $\ell^\infty(\X)$ be a coideal  of a discrete quantum group  $\G$. The following are equivalent:
    \begin{enumerate}
        \item $\ell^\infty(\X)$ is relatively amenable;
        \item there exists a state $\nu : C(\partial_F\G) \to\C$ such that $\ell^1(\G)P_{\X}\subseteq \Fix(\nu)$;
        \item there exists a state $\nu : C(\partial_F\G) \to\C$ such that $\fP_\nu(C(\partial_F\G)) \subseteq \ell^\infty(\X)$.
    \end{enumerate}
\end{lemma}
\begin{proof}
    Fix a $\G$-equivariant u.c.p. projection $\Psi : \ell^\infty(\G)\to C(\partial_F\G)$ and a $\G$-equivariant embedding $C(\partial_F\G)\hookrightarrow \ell^\infty(\G)$.
    
    (1)\!$\implies$\!(2) By Theorem \ref{RAcoideal},
    there exists a $\ell^1(\G)P_\X$-invariant state $m \in S(\ell^\infty(\G))$.
    We let  $\nu := m|_{C(\partial_F\G)}$ which  is obviously still $\ell^1(\G)P_\X$-invariant, i.e., satisfies $\ell^1(\G)P_{\X}\subseteq \Fix(\nu)$.
    
    (2)\!$\implies$\!(1) Assume that $\ell^1(\G)P_{\X}\subseteq \Fix(\nu)$ for some state $\nu : C(\partial_F\G)\to\C$. Then $m := \nu\circ \Psi$ is a $\ell^1(\G)P_\X$-invariant state, i.e.     for every $f\in \ell^1(\G)$,
    \[
   (fP_{\X})*m = (fP_{\X})*(\nu\circ\Psi) = (fP_{\X}*\nu)\circ\Psi = f(P_{\X})\nu\circ\Psi= fP_{\X}(1) m.
    \]
    (2)\!$\iff$\!(3) The statement follows immediately from Lemma \ref{Stationarity Lemma} and the observation that
    \begin{equation*}
    \ell^\infty(\X) = \{x\in \ell^\infty(\G) : f*x = f(1)x \quad \forall f\in \ell^1(\G)P_{\X}\}.
\qedhere
    \end{equation*}
\end{proof}

 Note that for a discrete quantum $\G$, and a Hilbert space $H$,  the von Neumann algebra $\ell^\infty(\G)\vtp B(H)$ equipped with the action $\Delta_\G\otimes \id$ is $\G$-injective \cite{MR4442841}*{Proposition 4.14}. In particular, for the Hilbert space $H=\C$ we have $\G$-injectivity of $\ell^\infty(\G)$. Now we will prove that for an arbitrary compact quasi-subgroup $\hh\X$ of $\hh\G$, $\ell^\infty(\X)$ is $\Gamma_\X$-injective, where $\Gamma_{\X} : \fO(\hh\G)\to B(\ell^2(\X))$ is the $\G$-representation associated with $\X$ and introduced in \eqref{Gamma}. This result generalizes \cite{MR4442841}*{Proposition 4.14} to the coideals associated with compact quasi-subgroups, where the Hilbert space is $H = \C$.

Let $\eta_\X : c_{00}(\X)\to \ell^2(\X)$ be the GNS map coming from the GNS construction of $\psi_{\X} : c_{00}(\X)\to \C$, the $\G$-invariant functional introduced in Section $2$. The $\G$-representation $\Gamma_\X$ satisfies
\[
\Gamma_\X(a)\eta_{\X}(x) = \eta_\X\left(x* S_{\hh\G}^{-1}(a)\right), \quad x\in c_{00}(\X), a\in \fO(\hh\G),
\]
if we use the identification $\fO(\hh\G)\cong \ell^1_F(\G)$.
\begin{lemma}\label{Quasiregular Representations Lemma}
    Let $\G$ be a discrete quantum group and $\hh\X$ be  a compact quasi-subgroup of $\hh\G$. If $a\in \fO(\hh\X\backslash\hh\G)$ or $a\in \fO(\hh\G/\hh\X)$ we have,
    \[
    \Gamma_\X(a)\eta_{\X}(P_\X) = \omega_{\hh\X}(a)\eta_{\X}(P_\X)
    \]
\end{lemma}
\begin{proof}
    Recall that $P_\X \ell^1_F(\G)\cong \fO(\hh\X\backslash\hh\G)$, and because $\fO(\hh\X\backslash\hh\G)$ is $\tau$-invariant, $S_{\hh\G}(\fO(\hh\X\backslash\hh\G)) = \fO(\hh\G/\hh\X) \cong \ell^1_F(\G)P_\X$. Then, for $f\in P_\X \ell^1(\G)$,
    \[
    \Gamma_\X(f)\eta_\X(P_\X) = \eta_\X( P_\X * S_{\hh\G}^{-1}(f) ) = f(P_\X)\eta_\X(P_\X).
    \]
    In particular $\Gamma_\X(a)\eta_\X(P_\X) = \omega_{\hh\X}(a)\eta_\X(P_\X)$ for every $a\in \fO(\hh\X\backslash\hh\G)$. The proof for $\fO(\hh\G/\hh\X)$ is similar and uses the fact that $\ell^1_F(\G)P_\X\cong \fO(\hh\G/\hh\X)$.
\end{proof}
Recall that for a $\G$-representation $\pi:\fO(\hh\G)\to B(H_\pi)$, the notion of $\pi$-injectivity of a $\G$-operator system $X$ can be found in Proposition \ref{Relative G-properties}. Note that if $X$ is a $\G$-invariant subspace of $B(H_\pi)$, then $\pi$-injectivity of $X$ is equivalent to the existence of a $\G$-equivariant  u.c.p. projection $E:B(H_\pi)\to X$.
\begin{proposition}\label{Injectivity of H_l}
    Let $\G$ be a discrete quantum group  and $\hh\X$ a compact quasi-subgroup of $\hh\G$. There exists a $\G$-equivariant u.c.p. projection $B(\ell^2(\X))\to \ell^\infty(\X)$. In particular, $\ell^\infty(\X)$ is $\Gamma_\X$-injective and there exists a $\G$-equivariant u.c.p. embedding $\B_{\Gamma_\X}\subseteq \ell^\infty(\X)$.
\end{proposition}
\begin{proof}
    Define the functional  $\omega_{P_\X}(T) =\langle T\eta_{\X}(P_\X),\eta_\X(P_\X) \rangle$, $T\in B(\ell^2(\X))$.
    
   Note that $\omega_{P_\X}(\Gamma_\X(1)) = \omega_{\hh\X}(1) = 1$, so,  $\omega_{\PG_\X}$ is a state. We claim that the Poisson transform $\fP_{\omega_{P_\X}} : B(\ell^2(\X))\to \ell^\infty(\G)$ is a $\G$-equivariant u.c.p. projection onto $\ell^\infty(\X)$. It is well-known that the Poisson transform of a state is a $\G$-equivariant u.c.p. map, so it is enough to   check that $\fP_{\omega_{P_\X}}$ is a projection onto $\ell^\infty(\X)$.
    Since for all $x\in\ell^\infty(\X)$,  $(1\otimes P_\X)\Delta_\G(x) = x\otimes P_\X$, we have that
    \[
    \fP_{\omega_{P_\X}}(x) = (\id\otimes \omega_{P_\X})\Delta_\G(x) = x
    \]
    Therefore, the proof will be complete once we show $\fP_{\omega_{P_\X}}(B(\ell^2(\X)))\subseteq \ell^\infty(\X)$. To prove this, let $f = \omega_{\xi}$ for some $\xi\in \ell^2(\G)$ and $T\in B(\ell^2(\X))$. Then, by Lemma \ref{Quasiregular Representations Lemma} and the fact  that $(\id\otimes\omega_{\hh\X})(\wW_\G) = P_\X$, we have
    \begin{align*}
        f (P_\X \fP_{\omega_{P_\X}}(T)) &= \langle (P_\X\otimes 1)U_{\Gamma_\X}^*(1\otimes T)U_{\Gamma_\X}(\xi\otimes P_\X), \xi\otimes P_\X\rangle
        \\
        &= \langle (P_\X\otimes 1)((\id\otimes \Gamma_\X)(\wW_\G))^*(1\otimes T)(\id\otimes \Gamma_\X)(\wW_\G))(\xi\otimes P_\X), \xi\otimes P_\X\rangle
        \\
        &= \langle (P_\X \otimes T)(\xi\otimes P_\X), \xi\otimes P_\X\rangle
        \\
        &= f(P_\X)\omega_{P_\X}(T).
    \end{align*}
    This shows $P_\X\fP_{\omega_{P_\X}}(T) = \omega_{P_\X}(T)P_\X$. Therefore,
    \begin{align*}
        (1\otimes P_\X)\Delta_\G( \fP_{\omega_{P_\X}}(T) ) &= (1\otimes P_\X)(\id\otimes \id\otimes \omega_{P_{\X}})(\Delta_\G\otimes\id)\ad_{\Gamma_\X}(T)
        \\
        &= (1\otimes P_\X)(\id\otimes \id\otimes \omega_{P_{\X}})(\id\otimes \ad_{\Gamma_\X})\ad_{\Gamma_\X}(T)
        \\
        &= (1\otimes P_\X)(\id\otimes \fP_{\omega_{P_\X}})\ad_{\Gamma_\X}(T)
        \\
        &= \fP_{\omega_{P_\X}}(T)\otimes P_\X.
    \end{align*}
    This proves $\fP_{\omega_{P_\X}}(T) \in \ell^\infty(\X)$ as desired.
\end{proof}

We have thus established sufficient material to prove the main theorem of this section.
\begin{theorem}\label{Relative Amenability iff FH-Boundaries Are Equal}
    Let $\G$ be a discrete quantum group and $\hh\X$ be a compact quasi-subgroup of $\hh\G$. We have that $\ell^\infty(\X)$ is relatively amenable if and only if $C(\partial_F\G) = \B_{\Gamma_\X}$.
\end{theorem}
\begin{proof}
    Assume $\ell^\infty(\X)$ is relatively amenable. By Lemma \ref{Relative Amenability Lemma}, there exists $\nu\in C(\partial_F\G)^*$ such that  $\fP_\nu :C(\partial_F\G)\to \ell^\infty(\X)$. The Poisson transform $\fP_\nu$ is a u.c.p. $\G$-equivariant map which is completely isometric by $\G$-essentiality of $C(\partial_F\G)$. By $\G$-injectivity and $\G$-rigidity, there is a $\G$-equivariant u.c.p. projection $B(\ell^2(\X)) \to C(\partial_F\G)$. Therefore, by construction of $\B_{\Gamma_\X}$,
    $$\B_{\Gamma_\X}\subseteq C(\partial_F\G).$$
    On the other hand, any $\G$-equivariant u.c.p. projection $B(\ell^2(\X))\to \B_{\Gamma_\X}$ restricts to a $\G$-equivariant u.c.p. projection $C(\partial_F\G)\to \B_{\Gamma_\X}$. By $\G$-essentiality, such a projection must be completely isometric and hence $C(\partial_F\G) = \B_{\Gamma_\X}$.
    
    Conversely, if $\B_{\Gamma_{\X}} = C(\partial_F\G)$ then, using Proposition \ref{Injectivity of H_l}, there is a $\G$-equivariant u.c.p. embedding $C(\partial_F\G)\subseteq \ell^\infty(\X)$ and we apply Lemma \ref{Relative Amenability Lemma}.
\end{proof}

As an application, we obtain the following corollary. Its conclusion is weaker than that of Theorem \ref{Coamenability Implies Amenability: CQSs}, however, it is interesting because the proof only requires an appropriate notion of a quasi-regular representation and $\Gamma_\X$-injectivity of $\ell^\infty(\X)$.
\begin{corollary}\label{Coamenability Implies Relative Amenability}
    Let $\G$ be a discrete quantum group and $\hh\X$ be a compact quasi-subgroup of $\hh\G$. If $\hh\X\backslash\hh\G$ is coamenable then $\ell^\infty(\X)$ is relatively amenable.
\end{corollary}
\begin{proof}
    Because of Proposition \ref{Categorical Property of FH Boundaries}, there exists a $\G$-equivariant u.c.p. map $C(\partial_F\G)=\B_{\lambda_\G}\to \B_{\Gamma_\X}$ which is completely isometric by $\G$-essentiality. By $\G$-injectivity and $\G$-rigidity, there is a $\G$-equivariant u.c.p. projection $\B_{\Gamma_\X} \to C(\partial_F\G)$. By construction, $\B_{\Gamma_\X}\subseteq C(\partial_F\G)$ and hence $C(\partial_F\G) = \B_{\Gamma_\X}$. Now apply Theorem \ref{Relative Amenability iff FH-Boundaries Are Equal}.
\end{proof}

\section{$\hh{\G}$-injectivity of Coideals}\label{sec5}
\subsection{Compact quantum group actions and $\hh{\G}$-injectivity}
\begin{definition}
  Let $\hh\G$ be a compact quantum group and $X$ be an operator system.
    \begin{itemize}
        \item We say $X$ is a (right) { $\hh\G$-\cst-operator system} if there is a unital completely isometric (u.c.i) map $\alpha_X : X\to X\otimes_{sp} C(\hh\G)$ such that $(\id\otimes\Delta_{\hh\G})\circ\alpha_X = (\alpha_X\otimes\id)\circ\alpha_X$ and
    $$\left[\alpha_X(X)\left(1\otimes C(\hh\G)\right) \right]= X {\otimes}_{sp} C(\hh\G).$$
    If $X = A$ is a unital \cst-algebra, then $\alpha_A$ is a $*$-homomorphism and we say $A$ is a { $\hh\G$-\cst-algebra}.

    \item We say $X$ is a { $\hh\G$-$W^*$-operator system} if there is a unital completely isometric (u.c.i) map $\alpha_X : X\to X\otimes_{\fF} L^\infty(\hh\G)$ such that $(\id\otimes\Delta_{\hh\G})\circ\alpha_X = (\alpha_X\otimes\id)\circ\alpha_X$.
    
    If $X = M$ is a von Neumann algebra, then we say $M$ is a { $\hh\G$-$W^*$-algebra}.
    \end{itemize}
    In either setting, a u.c.p. map $\Phi : X\to Y$ is called { $\hh\G$-equivariant} if $\alpha_Y\circ\Phi = (\Phi\otimes\id)\circ\alpha_X$.
\end{definition}
We will now discuss our main examples.

Let $\fO(\hh\X\backslash\hh\G)\subseteq \fO(\hh\G)$ be a coideal. Then its associated \cst-algebra $C(\hh\X\backslash\hh\G)$ and von Neumann algebra $L^\infty(\hh\X\backslash\hh\G)$ are respectively, a $\hh\G$-\cst-algebra and a $\hh\G$-$W^*$-algebra, where the actions are given by restriction of the comultiplication of ${\hh\G}$.

Let $A$ be a (left) $\G$-\cst-algebra with an action $\alpha$.
Define the reduced crossed product
$$\crosGAr := \overline{\Span}\{\alpha(x)(a\otimes 1) : x\in A, a\in C(\hh\G)\} \subseteq B(\ell^2(\G))\otimes A$$
and the Fubini crossed product
$$\crosGAF := \{Z\in B(\ell^2(\G))\otimes_{\fF} A : \left(\Delta^r_\G\otimes\id\right)(Z) = \left(\id\otimes \alpha\right)(Z)\}$$
where $\Delta^r_\G : B(\ell^2(\G))\to B(\ell^2(\G))\otimes \ell^\infty(\G)$ is the action defined by setting $\Delta^r_\G(T) = \vv_\G(T\otimes 1)\vv_\G^*$. It is clear  that
$L^\infty(\hh\G)\otimes 1, ~\alpha(A),$ and $\crosGAr$ are subspaces of  $\crosGAF$.
It turns out that $\crosGAr$  is a (right) $\hh\G$-\cst-algebra and $\crosGAF$ is a  (right) $\hh\G$-$W^*$-algebra  via the restriction of the following action \cite{dRH23}*{Theorem 3.37}
$$\hh\alpha: B(\ell^2(\G))\otimes_{\fF} A \to B(\ell^2(\G))\otimes_{\fF} A \otimes_{\fF} L^\infty(\hh\G), ~ \hh\alpha(Z) = (\vv_{\hh\G})_{13}Z_{12}(\vv_{\hh\G})^*_{13},$$
where $\vv_{\hh\G} \in \ell^\infty(\G)'\vtp L^\infty(\hh\G)$ is the right fundamental unitary of $\hh\G$. Moreover, we have
\begin{itemize}
    \item $\hh\alpha(\alpha(x)) = \alpha(x)_{12}$, $x\in A$;
    \item $\hh\alpha(a\otimes 1) = \left(\Delta_{\hh\G}(a)\right)_{13}$, $a\in L^\infty(\hh\G)$.
\end{itemize}
\begin{remark}
    Note that in \cite{dRH23} different conventions for duality and actions are being used. In particular, we are considering left (right) variants of their notions of $\G$-actions ($\hh\G$-actions) respectively and hence of their corresponding results.
\end{remark}

Let $\G$ be a discrete quantum group and $\hh\X$ be a compact quasi-subgroup of $\hh\G$. Let $L^2(\hh\X\backslash\hh\G) = L^2(L^\infty(\hh\X\backslash\hh\G), h_{\hh\G})$ be the subspace of $L^2(\hh\G)$ generated by $L^\infty(\hh\X\backslash\hh\G)$ via the GNS construction of $h_{\hh\G}$. It turns out that
$$P_{\X}L^2(\hh\G) = L^2(\hh\X\backslash\hh\G), ~ P_\X\eta_{\hh\G}(x) =(\omega_{\hh\X}\otimes\id)(\Ww_{\hh\G})\eta_{\hh\G}(x) = \eta_{\hh\G}(E_{\omega_{\hh\X}}(x)).$$
Consider the von Neumann algebra $P_\X B(L^2(\hh\G)) P_\X$ in $B(L^2(\hh\X\backslash\hh\G))$.
The restriction $\Delta^r_{\hh\G}|_{P_\X  B(L^2(\hh\G))P_\X}$ is an action of $\hh \G$ on the von Neumann algebra $P_\X B(L^2(\hh\G))P_\X$, implemented by $\vv_{\hh\G}\in \ell^\infty(\G)^\prime\vtp L^\infty(\hh\G)$. Indeed, since $P_\X\in \ell^\infty(\G)$, $$\Delta^r_{\hh\G}(P_\X(\cdot)P_\X) =\vv_{\hh\G}(P_\X(\cdot)P_\X)\vv_{\hh\G}^*= (P_\X\otimes 1)\Delta^r_{\hh\G}(\cdot)(P_\X\otimes 1).$$

More generally, given any $\hh\G$-$W^*$-operator system $X\subseteq B(L^2(\hh\G))$ where the action of $\hh\G$ on X is given by  the restriction of $\Delta^r_{\hh\G}$, the operator system $P_\X X P_\X \subseteq P_\X B(L^2(\hh\G)) P_\X$ is also a $\hh\G$-$W^*$-operator system.

\begin{remark}
    Note that the embedding $P_\X B(L^2(\hh\G))P_\X \subseteq B(L^2(\hh\G))$ is not necessarily unital. In particular, we cannot identify $P_\X B(L^2(\hh\G))P_\X$ as a $\hh\G$-$W^*$-subalgebra of $B(L^2(\hh\G))$.
\end{remark}

With the same observations, the map
\begin{align}\label{Projection PH}
    B(L^2(\hh\G))\to P_\X B(L^2(\hh\X\backslash\hh\G))P_\X, ~ T\mapsto P_\X TP_\X
\end{align}
is a $\hh\G$-equivariant u.c.p. surjection. Furthermore, because $P_\X x = x P_\X$ for all $x\in L^\infty(\hh\X\backslash\hh\G)$ and $x P_\X = 0 \implies x = 0$, the map
\begin{align}\label{Isomorphism H}
    L^\infty(\hh\X\backslash\hh\G)\to P_\X B(L^2(\hh\G))P_\X, ~ x\mapsto x P_\X
\end{align}
is a normal $\hh\G$-equivariant $*$-isomorphism onto its image.

Also, given $x,y\in L^\infty(\hh\G)$,
\begin{align*}
    P_\X x P_\X \eta_{\hh\G}(y) &= \eta_{\hh\G}(E_{\omega_{\hh\X}}(xE_{\omega_{\hh\X}}(y))) = E_{\omega_{\hh\X}}(x)P_\X\eta_{\hh\G}(y),
\end{align*}
which shows that $P_\X L^\infty(\hh\G) P_\X= L^\infty(\hh\X\backslash\hh\G)P_\X$.

To conclude this subsection, we will prove a result that comprises the main ingredient of the proof of the main result of this section. It is also potentially of independent interest as it is stated for coideals associated with arbitrary compact quasi-subgroups.
\begin{lemma}\label{Commutant Duality Lemma}
    Let $\G$ be a discrete quantum group and $\hh\X$ be a compact quasi-subgroup of $\hh\G$. We have that
    \begin{align*}
        P_\X L^\infty(\hh\G)' P_\X &=  {\{ (\id\otimes \mu)((P_\X\otimes P_\X)\vv_\G(P_\X\otimes 1)) : \mu\in B(\ell^2(\G))_* \}}^{\text{weak}^*-\text{closure}}.
    \end{align*}
    Also, $(P_\X L^\infty(\hh\G)' P_\X)' = L^\infty(\hh\X\backslash\hh\G)P_\X$.
\end{lemma}
\begin{proof}
For the first part, since  
\[
P_\X L^\infty(\hh\G)^\prime P_\X= {\{ (\id\otimes \mu)((P_\X\otimes 1)\vv_\G (P_\X\otimes 1): \mu\in B(\ell^2(\G))_* \}}^{\text{weak}^*-\text{closure}}.
\]
Then using the fact that $P_\X$ is a group-like projection we have
\begin{align*}
    (P_\X\otimes 1)\vv_\G (P_\X\otimes 1)&= (P_\X\otimes 1)\vv_\G (P_\X\otimes 1)(P_\X\otimes 1)\\
    &= (P_\X\otimes 1)\vv_\G (P_\X\otimes 1)\vv_\G^* \vv_\G (P_\X\otimes 1)\\
    &= (P_\X\otimes P_\X)\vv_\G (P_\X\otimes 1)
\end{align*}
For the second part, it is known that $L^\infty(\hh\G)^\prime=\hh{J}L^\infty(\hh\G) \hh{J}$ and $R_\G(P_\X)=\hh{J}P_\X\hh{J}=P_\X$, where $\hh{J}$ is the modular conjugation of $h_{\hh{\G}}$. Therefore, $\hh{J}P_\X=P_\X \hh{J}$ and we have
\begin{equation}\label{comPX}
P_\X L^\infty(\hh\G)' P_\X= \hh{J}\left( P_\X L^\infty(\hh\G) P_\X\right) \hh{J}= \hh{J}\left( L^\infty(\hh\X\backslash\hh\G) P_\X\right) \hh{J}
= P_\X \left( \hh{J} L^\infty(\hh\X\backslash\hh\G) \hh{J}\right)P_\X,
\end{equation}
where in the last equality we use the fact that $L^\infty(\hh\X\backslash\hh\G) =L^\infty(\hh\G)\cap \{P_\X\}^{\prime}$ and $\hh{J}P_\X=P_\X \hh{J}$. In particular, it is easy to observe that $\hh{J} \{P_\X\}^\prime \hh{J}= \{P_\X\}^\prime$ so we have that $\hh{J} L^\infty(\hh\X\backslash\hh\G) \hh{J} = L^\infty(\hh\G)^\prime \cap \{P_\X\}^\prime$.
Using \ref{comPX} and the fact that $P_\X\in \left( \hh{J} L^\infty(\hh\X\backslash\hh\G) \hh{J}\right)^\prime$ to find the commutant of $P_\X L^\infty(\hh\G)' P_\X$ in $P_\X B(L^2(\hh\G)) P_\X$,
\[
\left(P_\X L^\infty(\hh\G)' P_\X\right)^\prime = P_\X \left( \hh{J} L^\infty(\hh\X\backslash\hh\G) \hh{J}\right)^\prime P_\X = P_\X \left( L^\infty(\hh\G) \vee \{P_\X\}^{\prime\prime} \right) P_\X= L^\infty(\hh\X\backslash\hh\G)P_\X,  
\]
where  in the last equality we used \cite{MR2335776}*{Lemma 3.3}.
\end{proof}

\begin{theorem}\label{Projection Lemma}
    Let $\G$ be a discrete quantum group and $\hh\X$ be a compact quasi-subgroup of $\hh\G$.
    \begin{enumerate}
        \item There exists a $\hh\G$-equivariant u.c.p. projection $\crosGBF \to L^\infty(\hh\X\backslash\hh\G)$ onto $L^\infty(\hh\X\backslash\hh\G)$.
        
        \item There exists a $\hh\G$-equivariant u.c.p. projection $\crosGBr \to C(\hh\X\backslash\hh\G)$ onto $C(\hh\X\backslash\hh\G)$.
    \end{enumerate}
\end{theorem}
\begin{proof}
    $(1)$ Let $L : B(\ell^2(\G))\to P_\X B(\ell^2(\G)) P_\X$, $Z\mapsto P_\X ZP_\X$ be the $\hh\G$-equivariant u.c.p. projection in  \eqref{Projection PH}. Then $\Psi = L\otimes\id : B(\ell^2(\G))\vtp B(\ell^2(\G))\to P_\X B(\ell^2(\G))P_\X\vtp B(\ell^2(\G))$ is a $\hh\G$-equivariant u.c.p. projection. Choose a state $\mu \in S(\B_{\Gamma_\X})$ such that $\fP_\mu(\B_{\Gamma_\X}) \subseteq \ell^\infty(\X)$, which is possible thanks to Proposition~\ref{Injectivity of H_l}.
    
    It is straightforward to see that any slice map $\id\otimes\mu$ on $B(\ell^2(\G))\otimes_{\fF}\B_{\Gamma_\X}$ is $\hh\G$-equivariant. Hence $\Psi_\mu := (\id\otimes \mu)\circ\Psi|_{\crosGBF}$ is $\hh\G$-equivariant.

    Take $Z\in \crosGBF$ and observe that
    \begin{align*}
        (1\otimes P_\X)\Delta^r_\G(\Psi_\mu(Z)) &= (\id\otimes\id\otimes\mu)
        \left(
        (P_\X\otimes P_\X\otimes 1) \left((\Delta^r_\G\otimes\id)Z\right)
        (P_\X\otimes 1\otimes 1)
        \right)
        \\
        &= (\id\otimes\id\otimes\mu)\left(
        (P_\X\otimes P_\X\otimes 1)
        \left( (\id\otimes\alpha_{\Gamma_\X})Z\right) (P_\X\otimes 1\otimes 1)\right)
        \\
&= (P_\X\otimes 1\otimes 1) \left( (\id\otimes P_\X \otimes \mu) (\id\otimes \alpha_{\Gamma_\X})Z \right)(P_\X\otimes 1\otimes 1)
        \\&=(P_\X\otimes 1) \left( ( 1\otimes P_\X) (\id\otimes \mu)Z \right)(P_\X\otimes 1)
        \\
        &= (\id\otimes\id\otimes \mu)((P_\X\otimes P_\X\otimes 1)(Z(P_\X\otimes 1))_{13})
        \\
        &= \Psi_\mu(Z)\otimes P_\X.
    \end{align*}
    Hence, upon multiplying the above equation by $(P_\X\otimes 1)$ from the left and remembering that $\Psi_\mu(Z)\in P_\X B(\ell^2(\G)) P_\X $ so $\Psi_\mu(Z) P_\X = P_\X\Psi_\mu(Z) = \Psi_\mu(Z)$,
    $$(P_\X\otimes P_\X)V_\G(\Psi_\mu(Z)\otimes 1)V_\G^* = (P_\X\otimes 1)(\Psi_\mu(Z)\otimes P_\X)$$
    and therefore
    $$(P_\X\otimes P_\X) \vv_\G (\Psi_\mu(Z)\otimes 1) = (\Psi_\mu(Z)\otimes P_\X)\vv_\G.$$
    Multiplying the above equation by $(P_\X\otimes 1)$ from the right we have
    \[
(P_\X\otimes P_\X) \vv_\G (\Psi_\mu(Z)\otimes 1) = (\Psi_\mu(Z)\otimes P_\X)(\vv_\G)(P_\X\otimes 1).
    \]
    Then, using Lemma \ref{Commutant Duality Lemma}, $\Psi_\mu(Z)\in (P_\X L^\infty(\hh\G)' P_\X)' = L^\infty(\hh\X\backslash\hh\G)P_\X$.
    
    Also, for $a\in L^\infty(\hh\X\backslash\hh\G)$, $\Psi_\mu(a\otimes 1) = aP_\X$, and we conclude that
    $$L^\infty(\hh\X\backslash\hh\G)P_\X \subseteq \Psi_\mu(\crosGBF)\subseteq P_\X L^\infty(\hh\G) P_\X.$$
    Finally, if we let $\sigma : L^\infty(\hh\X\backslash\hh\G) \to L^\infty(\hh\X\backslash\hh\G)P_\X$ be the restriction of the $\hh\G$-equivariant $*$-isomorphism in \eqref{Isomorphism H},
    $$\sigma^{-1}\circ \Psi_\mu : \crosGBF \to L^\infty(\hh\X\backslash\hh\G), ~ \alpha_{\Gamma_\X}(x)(a\otimes 1)\mapsto \mu(x)E_{\omega_{\hh\X}}(a)$$
    is the desired $\hh\G$-equivariant u.c.p. projection.

    $(2)$ Since $(P_\X\otimes \mu)\alpha_{\Gamma_\X}(x) = P_\X\mu(x)$,
    $$(\id\otimes\mu)\Psi\left(\alpha_{\Gamma_\X}(x)(a\otimes 1)\right) = P_\X a P_\X\mu(x), ~ a\in C(\hh\G), x\in \B_{\Gamma_\X}.$$
    It follows from linearity and density that the restriction
    $\Psi_\mu := (\id\otimes\mu)\Psi |_{\crosGBr}$ is a $\hh\G$-equivariant u.c.p. map onto $P_\X C(\hh\G)P_\X  = C(\hh\X\backslash\hh\G)P_\X$. Now proceed as in the last part of the proof of $(1)$.
\end{proof}
\begin{remark}
    We note that Theorem \ref{Projection Lemma} is new, even for groups.  Let $\G = G$ be a discrete group. Recall $ \fO(\hh \G) = \C[G]$ and $L^\infty(\hh\G)=\mathcal{L}(G)$ is the group von Neumann algebra. Any coideal of $\hh\G$ is of the form $\fO(\hh\X\backslash \hh \G) = \C[H]$ for some subgroup $H\leq G$ and $\hh\X=\widehat{G/H}$ is a compact quasi-subgroup.
    \end{remark}

\subsection{$ \hh{\G}$-$W^*$-injectivity and $\G$-injectivity of coideals}
In this subsection, we will prove that for a compact quasi-subgroup $\hh\X$ of $\hh\G$, $\G$-injectivity of $\ell^\infty(\X)$ implies $\hh\G$-$W^*$-injectivity of $L^\infty(\hh\X\backslash \hh\G)$. Moreover, we have a partial converse where we must replace $\ell^\infty(\X)$ with the $\G$-operator system $M_{P_\X}$ and $\G$-injectivity with relative amenability. This generalizes \cite{dRH23}*{Corollary 4.4} to coideals associated with compact quasi-subgroups.
We note in particular that $\hh\G$-$W^*$-injectivity of a coideal $L^\infty(\hh\X\backslash \hh\G)$ implies $\hh\X$ is a compact quasi-subgroup. This shows how our methods do not apply to cases beyond the compact quasi-subgroups.

\begin{definition} Let $\hh\G$ be a compact quantum group and $X$ be a $\hh\G$-\cst-operator system ($\hh\G$-$W^*$-operator system).
    We say $X$ is { $\hh\G$-\cst-injective ($\hh\G$-$W^*$-injective)} if for every $\hh\G$-\cst-operator systems ($\hh\G$-$W^*$-operator system) $Y_1$ and $Y_2$,  every $\hh\G$-equivariant u.c.p. map $\varphi : Y_1\to X$, and every $\hh\G$-equivariant u.c.i map $\iota : Y_1\to Y_2$, there exists a $\hh\G$-equivariant u.c.p. map $\tilde{\varphi}: Y_2\to X$ such that $\tilde\varphi\circ\iota = \varphi$.
\end{definition}

\begin{proposition}\label{C*Injectivity Implies Relative Amenability}
    Let $\G$ be a discrete quantum group and $\fO(\hh\X\backslash\hh\G)$ be a coideal subalgebra of $\hh\G$. If $C(\hh\X\backslash\hh\G)$ is $\hh\G$-\cst-injective then $\hh\X$ is a compact quasi-subgroup and $M_{P_\X}$ is relatively amenable.
\end{proposition}
\begin{proof}
    If $C(\hh\X\backslash\hh\G)$ is $\hh\G$-\cst-injective then there exists a $\hh\G$-equivariant u.c.p. projection $C(\hh\G) \to C(\hh\X\backslash\hh\G)$ onto $C(\hh\X\backslash\hh\G)$. A $\hh\G$-equivariant map is easily seen to be $h_{\hh\G}$-preserving, hence $\hh\X$ is a compact quasi-subgroup by \cite{MR3667214}*{Theorem 4.5}.

    There exists a $\hh\G$-equivariant u.c.p. projection $\Psi : \crosGFr \to C(\hh\X\backslash\hh\G)$. The restriction $\Psi|_{\alpha_F(C(\partial_F\G))} : \alpha_F(C(\partial_F\G)) \to C(\hh\X\backslash\hh\G)$ is still $\hh\G$-equivariant and
    $$\Delta_{\hh\G}(\Psi(\alpha_F(a))) = (\Psi\otimes \id)\hat\alpha_F(\alpha_F(a))= (\Psi\otimes \id)(\alpha_F(a)_{12}) =\Psi(\alpha_F(a))\otimes 1, ~ a\in C(\partial_F\G).$$
    Hence $\Psi(\alpha_F(C(\partial_F\G))) \subseteq \C$ by \cite{MR1832993}*{Result 5.13}. Then for $f\in \ell^1(\G)$ and $T\in \crosGFr$,
    \begin{align*}
        &\Psi(((P_\X fP_\X)\otimes\id)\Delta_\G(T))
        \\
        &= (f\otimes \Psi)((P_{\X}\otimes 1)\ww^*_\G(1\otimes T)\ww_\G(P_\X\otimes 1))
        \\
        &= (f\otimes \id)(((P_{\X}\otimes 1)\ww^*_\G)(1\otimes \Psi(T))\ww_\G(P_\X\otimes 1)) ~ \text{(since $C(\hh\X\backslash\hh\G)\subseteq\mathrm{mult}(\Psi)$)}
        \\
        &= ((P_\X f P_\X)\otimes\id)\Delta_\G(\Psi(T))
        \\
        &= \Psi(T)f(P_\X)
    \end{align*}
    where $\mathrm{mult}(\Psi)$ is the multiplicative domain of $\Psi$. Therefore, if we choose any $\G$-equivariant u.c.p. projection $\Phi : \ell^\infty(\G)\to C(\partial_F\G)$, the composition $m = \Psi \circ \alpha_F\circ \Phi$
    satisfies $(P_\X f P_\X)*m = f(P_\X)m$ for $f\in \ell^1(\G)$ and the conclusion follows from \cite{AS21(1)}*{Proposition 3.8}.
\end{proof}

Let $X$ and $Y$ be $\hh\G$-$W^*$-operator systems ($\hh\G$-\cst-operator systems) such that there exists a $\hh\G$-equivariant complete isometry $X\to Y$. Suppose $Y$ is $\hh\G$-$W^*$-injective (respectively $\hh\G$-\cst-injective). It is well-known that $X$ is $\hh\G$-$W^*$-injective (respectively $\hh\G$-\cst-injective) iff there exists a $\hh\G$-equivariant u.c.p. projection $Y\to X$.

\begin{theorem}\label{Relative Amenability and Injectivity Theorem}
    Let $\G$ be a discrete quantum group and $\hh\X$ be a compact quasi-subgroup of $\hh\G$. Consider the following statements
    \begin{enumerate}
        \item $\ell^\infty(\X)$ is $\G$-injective.
        \item $\ell^\infty(\X)$ is relatively amenable.
        \item $L^\infty(\hh\X\backslash\hh\G)$ is $\hh\G$-$W^*$-injective.
        \item $C(\hh\X\backslash\hh\G)$ is $\hh\G$-\cst-injective.
    \end{enumerate}
    Then $(1)\!\Longrightarrow\!(2)\!\Longrightarrow\!(3)\!\Longrightarrow\!(4)$. Moreover, if $\ell^\infty(\X)=M_{P_\X}$ we have $(4)\!\Longrightarrow\!(2)$.
\end{theorem}
\begin{proof}
 $(1)\!\Longrightarrow\!(2)$: This is trivial.  
  $(2)\!\Longrightarrow\!(3)$:  Since $\ell^\infty(\X)$ is relatively amenable, $\B_{\Gamma_\X} = C(\partial_F\G)$ by Theorem \ref{Relative Amenability iff FH-Boundaries Are Equal}.
  Therefore, by Lemma \ref{Projection Lemma}, there is a $\hh\G$-equivariant u.c.p. projection $\crosGFF \to L^\infty(\hh\X\backslash\hh\G)$ onto $L^\infty(\hh\X\backslash\hh\G)$. Since $C(\partial_F \G)$ is $\G$-injective, we have $\hh\G$-$W^*$-injectivity of $\crosGFF$ by \cite{dRH23}*{Theorem 4.3}.
  Hence, $L^\infty(\hh\X\backslash\hh\G)$ is $\hh\G$-$W^*$-injective.
  $(3)\!\Longrightarrow\!(4)$: Borrowing the notation from \cite{dRH23}, it is easy to check that $\mathcal{R}(L^\infty(\hh\X\backslash\hh\G)) = C(\hh\X\backslash\hh\G)$. Therefore, we have the result by  \cite{dRH23}*{Lemma 3.16}. The last part is proved in Proposition \ref{C*Injectivity Implies Relative Amenability}.
\end{proof}
Consider the case where $\ell^\infty(\X) = \ell^\infty(\G/\h)$ is a coideal of quotient type. Recall from \cite{MR4442841}*{Theorem 3.7} that $\h$ is amenable iff $\ell^\infty(\G/\h)$ is relatively amenable iff $\hh\h$ is coamenable. In this special case, we obtain a characterization of relative amenability in terms of $\hh\G$-$W^*$-injectivity.
\begin{corollary}
    Let $\G$ be a discrete quantum group and $\h\leq \G$ a closed quantum subgroup. The following are equivalent:
    \begin{enumerate}
    \item $\h$ is amenable;
      \item $C(\hh\h)$ is $\hh\G$-\cst-injective.
     \item $L^\infty(\hh\h)$ is $\hh\G$-$W^*$-injective;
       \item $\crosGGHr$ is $\hh\G$-\cst-injective;
        \item $\crosGGHF$ is $\hh\G$-$W^*$-injective;
    \end{enumerate}
\end{corollary}
\begin{proof}
$(1)\!\Longleftrightarrow\!(2)\!\Longleftrightarrow\!(3)$ follows from  \cite{dRH23}*{Corollary 4.4}.
    $(1)\!\Longrightarrow\!(5)$ follows because $\B_{\Gamma_{\G/\h}} = C(\partial_F\G)$ thanks to Theorem \ref{Relative Amenability iff FH-Boundaries Are Equal}.
    $(5)\!\Longrightarrow\!(3)$ follows from Lemma \ref{Projection Lemma}.
    $(4)\!\Longleftrightarrow\!(5)$ follows from \cite{dRH23}*{Theorem 4.3}.
\end{proof}

\section{Unique trace property}\label{sec6}
The notion of amenable actions of discrete quantum group $\G$ on a von Neumann algebra $N$ is defined and studied in \cite{M18}.  An action $\alpha:\G\act  N$ is called amenable if there exists
a conditional expectation $E_\alpha : N\vtp \ell^\infty(\G) \to \alpha(N)$ such that $(\id\otimes \Delta_\G)\circ E_\alpha=(E_\alpha\otimes\id)\circ(\id\otimes\Delta_\G)$. In the next theorem we are dealing with the
(right) crossed product of (right) coideals $\ell^\infty(\X_r)$ with respect to the comultiplication of $\G$.  We denote it by $\ell^\infty(\X_r)\rtimes \G $ and is the following von Neumann algebra,
\[
\ell^\infty(\X_r)\rtimes \G :=\{ \Delta_\G(\ell^\infty(\X_r)) \cup \C\vtp L^\infty(\hh\G)^\prime \}^{\prime\prime}.
\]
\begin{theorem}\label{Injectivity Characterizations}
    Let $\G$ be a discrete quantum group. Let $\ell^\infty(\X_r)$ be a (right) von Neumann algebra coideal of $\G$. Then the following statements are equivalent:
    \begin{enumerate}
   
    \item The action $\G\act \ell^\infty(\X_r)$ is amenable and $\ell^\infty(\X_r)$ is an injective von Neumann algebra.
       \item $\ell^\infty(\X_r)$ is $\G$-injective.
    \end{enumerate}
    Moreover, if $\G$ is unimodular  then the latter statements are equivalent with the following:
    \begin{enumerate}[resume]
    \item $\ell^\infty(\X_r)  \rtimes \G$ is injective;
     \item $L^\infty(\hh\X\backslash\hh\G)$ is injective.
    \end{enumerate}
\end{theorem}
\begin{proof}
The proof of this theorem is based on some known results about amenable actions of discrete quantum groups on von Neumann algebras \cite{M18}. $(1)\!\Longrightarrow\!(2)$: Since the action of a discrete quantum group $\G$ on a coideal $\ell^\infty(\X_r)$ is the restriction of the comultiplication of $\G$, the result is straightforward by \cite{M18}*{Definition 4.4.(1),  Proposition 4.5.(1)}. $(2)\!\Longrightarrow\!(1)$: By assumption there exists a $\G$-equivariant conditional expectation $\ell^\infty(\G)\to \ell^\infty(\X_r)$, which yields injectivity of $\ell^\infty(\X_r)$. Moreover, the action $\G\act \ell^\infty(\X_r)$ is amenable by \cite{M18}*{Proposition 4.5.(2)}.

 If $\G$ is unimodular  we have $(1)\!\iff\!(3)$ by \cite{M18}*{Corollary 6.4}.
$(3)\!\iff\!(4)$: The crossed product $\ell^\infty(\X_r)  \rtimes \G$ can be identified with the von Neumann algebra $M=(\ell^\infty(\X_r) \cup L^\infty(\hh\G)^\prime)^{\prime\prime}$ \cite{embed}*{Remark 3.8}. Therefore, injectivity of $\ell^\infty(\X_r)  \rtimes \G$ is equivalent to injectivity of $M$ and also $M^\prime$. By the definition of codual coideals we have,
\[
L^\infty(\hh\X\backslash \hh\G) = \ell^\infty(\X_r)^\prime \cap L^\infty (\hh\G) =M^\prime,
\]
which yields the equivalence of injectivity of $\ell^\infty(\X_r)  \rtimes \G$ with injectivity of $L^\infty(\hh\X\backslash \hh\G)$.
\end{proof}
The following theorem states that relative amenability is equivalent to $\G$-injectivity for all coideals  $\ell^\infty(\X)$ of a unimodular discrete quantum group $\G$ under the condition that $\hh\X$ is compact quasi-subgroup of $\hh\G$. Actually, this is a noncommutative version  of \cite{MR3267522}*{Theorem 2.(a)} for unimodular discrete quantum groups.
\begin{theorem}\label{Relative Amenability Implies GInjectivity}
    Let $\G$ be a unimodular discrete quantum group. Let $\hh\X$ be a compact quasi-subgroup of $\hh\G$. If  $\ell^\infty(\X)$ is relatively amenable, then $\ell^\infty(\X)$ is $\G$-injective.
\end{theorem}
\begin{proof}
    Since $\ell^\infty(\X)$ is relatively amenable, $L^\infty(\hh\X\backslash \hh\G)$ is $\hh\G$-$W^*$-injective by
 Theorem~\ref{Relative Amenability and Injectivity Theorem}. Moreover, $B(L^2(\hh\G))$ is a $\hh\G$-von Neumann algebra where the action is implemented by the (left) multiplicative unitary. The embedding $L^\infty(\hh\X\backslash \hh\G) \subseteq B(L^2(\hh\G))$ is $\hh\G$-equivariant so there exists a ($\hh\G$-equivariant) conditional expectation $E:B(L^2(\hh\G))\to L^\infty(\hh\X\backslash \hh\G)$ onto $L^\infty(\hh\X\backslash \hh\G)$, which in particular, implies $L^\infty(\hh\X\backslash \hh\G)$ is injective von Neumann algebra. Hence, by Theorem \ref{Injectivity Characterizations} the (right) coideal $\ell^\infty(\X_r)$ is $\G$-injective which is equivalent to $\G$-injectivity of the (left) coideal
 $R_\G(\ell^\infty(\X_r))=\ell^\infty(\X)$.
\end{proof}

The next proposition generalizes \cite{MR3693148}*{Proposition 5.5} to $\G$-$W^*$-injective coideals of a locally compact quantum group $\G$.
\begin{proposition}\label{CB=CP} Let  $M$ be a $\G$-$W^*$injective (right) coideal von Neumann algebra of a locally compact quantum group $\G$, we have
\[
CB_{\G}(L^\infty(\G), M)= span CP_{\G}(L^\infty(\G), M).
\]
\end{proposition}
\begin{proof}We skip the proof since it is similar to the proof of  \cite{MR3693148}*{Proposition 5.5}. We just need to note that $M$ is an injective von Neumann algebra since $M$ is a $\G$-invariant subalgebra of $B(L^2(\G))$ and $M$ is $\G$-$W^*$-injective.
\end{proof}

Let $\hh\G$ be a compact quantum group. The half-lifted versions of the comultiplication of $\hh\G$  will be denoted by $\Delta_r^{u,r}$, and $\Delta_r^{r,u}$, \cite{MR3667214}. In fact,
\begin{align*}
   \Delta_r^{u,r}:& C(\hh\G)\to C^u(\hh\G)\otimes C(\hh\G), \quad (\id\otimes\Lambda_{\hh\G})\circ \Delta^u_{\hh\G}=  \Delta_r^{u,r} \circ \Lambda_{\hh\G}\\
   \Delta_r^{r,u}:& C(\hh\G)\to C(\hh\G)\otimes C^u(\hh\G), \quad (\Lambda_{\hh\G}\otimes\id)\circ \Delta^u_{\hh\G}=  \Delta_r^{r,u} \circ \Lambda_{\hh\G}
\end{align*}
It can be easily observed that
\begin{equation}\label{halflifted}
    (\Delta_r^{r.u}\otimes\id)\circ\Delta_{\hh\G}=(\id\otimes\Delta_r^{u,r})\circ\Delta_{\hh\G}.
\end{equation}

\begin{corollary}\label{CBandfunctional}
Let $\G$ be a discrete quantum group.
 Assume that the (right) coideal  $L^\infty(\hh\X\backslash \hh\G)$ is $\hh\G$-$W^*$-injective. Then for every completely bounded $\hh\G$-equivariant map $\Phi:L^\infty(\hh\G) \to L^\infty(\hh\X\backslash \hh\G)$ there exists a functional $\mu\in C^u(\hh\G)^*$
such that
\[
\Phi(x)= (\mu\otimes\id)\Delta_r^{u,r}(x)=(\mu\otimes \id)\left(\Ww^*_\G (1\otimes x) \Ww_\G\right), \quad x\in L^\infty(\hh\G).
\]
\end{corollary}
\begin{proof}
This proof is easy and is based on \cite{MR2500076}*{Theorem 4.1}, \cite{MR3693148}*{Lemma 5.7} and Proposition \ref{CB=CP}.
\end{proof}
The next theorem along with Theorem \ref{Coamenability Implies Amenability: CQSs} generalizes the fact that a compact quantum group $\hh\G$ is coamenable if and only if $\G$ is amenable ($\C$ is $\G$-injective) to coideals associated with a compact quasi-subgroup. The proof of the following theorem is a modification of \cite{MR3693148}*{Theorem 5.10}.
\begin{theorem}\label{coam}
 Let $\G$ be a discrete quantum group and $\hh\X$ be a compact quasi-subgroup of $\hh\G$.  If $\ell^\infty(\X)$ is $\G$-injective then $\hh\X\backslash\hh\G$ is a coamenable coideal.
\end{theorem}
\begin{proof}
According to  Theorem~\ref{Relative Amenability and Injectivity Theorem}, $\G$-injectivity of $\ell^\infty(\X)$ ensures $\hh\G$-$W^*$-injectivity  of the von Neumann algebra $L^\infty(\hh\X\backslash \hh\G)$. Let us denote $\alpha:=\Delta_{\hh\G}|_{L^\infty(\hh\X\backslash\hh\G)}$. The $\hh\G$-equivariant left-inverse of $\alpha$, which exists by $\hh\G$-$W^*$-injectivity of $L^\infty(\hh\X\backslash\hh\G)$,  will be denoted  by $\Phi$, i.e.
\[
\Phi:L^\infty(\hh\X\backslash \hh\G)\vtp L^\infty(\hh\G) \to L^\infty(\hh\X\backslash \hh\G),  \quad (\Phi\otimes\id)\circ(\id\otimes\Delta_{\hh\G})= \alpha\circ \Phi, \quad \Phi\circ\alpha= \id.
\]
    As a unital complete contraction, $\Phi$ is completely positive and $\Phi|_{C(\hh\X\backslash\hh\G)\otimes C(\hh\G)} \neq 0$
    since $C(\hh\G)$ and $C(\hh\X\backslash\hh\G)$ are unital \cst-algebras. Moreover, the \cst-algebra $C(\hh\X\backslash\hh\G)$  is nuclear by  Theorem~\ref{GInjective Implies Nuclear}. Therefore, there exists a net $(\Psi_a)_{a\in A}:C(\hh\X\backslash\hh\G)\to C(\hh\X\backslash\hh\G)$   of finite-rank, u.c.p. maps converging to the identity map $\id:C(\hh\X\backslash\hh\G)\to C(\hh\X\backslash\hh\G)$ in the point-norm topology.
For $a\in A$,    consider the unital completely positive map $\Phi_a:C(\hh\X\backslash\hh\G)\to C(\hh\X\backslash\hh\G)$ given by    $$\Phi_a=\Phi\circ(\Psi_a\otimes\id)\circ\alpha|_{C(\hh\X\backslash\hh\G)}.$$
    
    Since $\Psi_a$ is finite rank, there exist $x^a_1,...,x^a_{n_a}\in C(\hh\X\backslash\hh\G)$ and $\mu^a_1,...,\mu^a_{n_a}\in C(\hh\X\backslash\hh\G)^*$ such that
    $$\Psi_a(x)=\sum_{n=1}^{n_a}\la\mu^a_n,x\ra x^a_n, \ \ \ x\in C(\hh\X\backslash\hh\G), \ a\in A.$$
    For each $a\in A$, and $1\leq n\leq n_a$, let $\Phi_{(a,n)}:L^\infty(\hh\G)\rightarrow L^\infty(\hh\X\backslash\hh\G)$ be defined by
    
    $$\Phi_{(a,n)}(x)=\Phi(x^a_n\otimes x), \ \ \ x\in L^\infty(\hh\G).$$
    Then $\Phi_{(a,n)}$ is completely bounded with $\norm{\Phi_{(a,n)}}_{cb}\leq\norm{x^a_n}_{C(\hh\X\backslash\hh\G)}$, and
\[
(\Phi_{(a,n)}\otimes \id)\Delta_{\hh\G}=\alpha\circ \Phi_{(a,n)}=\Delta_{\hh\G}\circ \Phi_{(a,n)}
\]    
where the first equality holds according to $\hh\G$-equivariance of $\Phi$.
So there exists $\nu^a_n\in C^u(\hh\G)^*$ such that $\Phi_{(a,n)}(x)=(\nu^a_n\otimes\id)\Delta_r^{u,r}$ by Corollary \ref{CBandfunctional}. Moreover, this shows that $\Phi_{(a,n)}(C(\hh\G))\subseteq C(\hh\X\backslash\hh\G)$.

Let $x\in C(\hh\X\backslash\hh\G)$. Then
\begin{align*}
\Phi_a(x)&=\Phi((\Psi_a\otimes\id)(\alpha(x)))\\
&=\sum_{n=1}^{n_a}\Phi\left( x^a_n\otimes \left( (\mu^a_n\otimes\id)\alpha(x)\right)\right)\\
&=\sum_{n=1}^{n_a}\Phi_{(a,n)}((\mu^a_n\otimes\id)\alpha(x))\\
&=\sum_{n=1}^{n_a} (\mu_n^a\otimes\nu^a_n\otimes \id)(\id\otimes \Delta_r^{u,r})\alpha(x)\\
&=\sum_{n=1}^{n_a} (\mu_n^a\otimes\nu^a_n\otimes \id)(\Delta_r^{r,u}|_{C(\hh\X\backslash\hh\G)}\otimes\id)\alpha(x)\\
&=\sum_{n=1}^{n_a}( \mu^a_n\star\nu^a_n \otimes\id)\alpha(x)
\end{align*}    
Letting $\mu_a=\sum_{n=1}^{n_a}\mu^a_n\star\nu^a_n$, where
    \[
 \mu^a_n\star\nu^a_n=(\mu^a_n\otimes\nu^a_n)\Delta^{r,u}_{r}|_{C(\hh\X\backslash\hh\G)}\in C(\hh\X\backslash\hh\G)^*,
 \]
 we obtain $\Phi_a=(\mu_a\otimes\id)\alpha$. Moreover, for every  $a\in A$,
$\norm{\mu_a}_{C(\hh\X\backslash\hh\G)^*}=\norm{\Phi_a}_{cb}=1.$
Since $\Phi_a$ converges to the identity map on $C(\hh\X\backslash\hh\G)$ in the point-norm topology it follows that
    $$x \star \mu_a=(\mu_a\otimes \id)\alpha(x)\rightarrow x, \ \ \  x\in C(\hh\X\backslash\hh\G).$$
    Let $\mu$ be a weak$^*$ cluster point of $(\mu_a)_{a\in A}$ in the unit ball of $C(\hh\X\backslash\hh\G)^*$. Then for every $\nu\in C(\hh\G)^*$, $\mu\star\nu=(\mu\otimes \nu)\alpha=\nu$. Moreover, if $\Lambda_{\hh\G}|_{C^u(\hh\X\backslash\hh\G)}:C^u(\hh\X\backslash\hh\G)\to C(\hh\X\backslash\hh\G)$ is the reducing morphism of $\hh\G$, then
    \[
    \alpha\circ\Lambda_{\hh\G}|_{C^u(\hh\X\backslash\hh\G)}=((\Lambda_{\hh\G}|_{C^u(\hh\X\backslash\hh\G)})\otimes \Lambda_{\hh\G})\circ\Delta_{\hh\G}^u.
    \]
    Let $\alpha^{u,r}:= (\id\otimes\Lambda_{\hh\G})\Delta_{\hh\G}^u$, then  $\alpha^{u,r} : C^u(\hh\X\backslash\hh\G)\to C^u(\hh\X\backslash\hh\G)\otimes C(\hh\G)$ and $\alpha^{u,r}(C^u(\hh\X\backslash\hh\G))(\I\otimes C(\hh\G))=C^u(\hh\X\backslash\hh\G)\otimes C(\hh\G)$. In other words, $\{\nu\star x =(\id\otimes\nu)\alpha^{u,r}(x)\st \nu\in C(\hh\G)^*, x\in C^u(\hh\X\backslash\hh\G)\}^{\text{norm-closure}}=C^u(\hh\X\backslash\hh\G)$. For every $\nu\in C(\hh\G)^*, x\in C^u(\hh\X\backslash\hh\G)$,
    \begin{align*}
    \mu\left(\Lambda_{\hh\G}|_{C^u(\hh\X\backslash\hh\G)}(\nu\star x)\right)&=
(\mu\circ\Lambda_{\hh\G}|_{C^u(\hh\X\backslash\hh\G)}\otimes \nu )  \alpha^{u,r} (x) =(\mu\otimes\nu)\alpha(x)\\
    &=\nu(x)=(\varepsilon_{\hh\G}|_{C^u(\hh\X\backslash\hh\G)}\otimes \nu)\alpha^{u,r}=\varepsilon_{\hh\G}|_{C^u(\hh\X\backslash\hh\G)}(\nu\star x).
    \end{align*}
    This implies that $\mu\circ\Lambda_{\hh\G}|_{C^u(\hh\X\backslash\hh\G)}=\varepsilon_{\hh\G}|_{C^u(\hh\X\backslash\hh\G)}$. So $\mu$ is the reduced version of the counit on $C(\hh\X\backslash\hh\G)$.
\end{proof}

\begin{remark}\label{Relative Amenability Implies Coamenability Remark}
    Let $\G$ be a discrete quantum group and $\hh\X$ a compact quasi-subgroup of $\hh\G$. Consider the following additional assumptions:
    \begin{enumerate}
        \item $\G$ is unimodular;
        \item $C(\hh\X\backslash\hh\G)$ is an exact \cst-algebra.
    \end{enumerate}
    In either case  $(1)$ or $(2)$ if $\ell^\infty(\X)$ is relatively amenable then we have that $\hh\X\backslash\hh\G$ is coamenable. Let us prove the above claims.
    \begin{enumerate}
        \item Suppose $\G$ is unimodular. Then, by Theorem~\ref{Relative Amenability Implies GInjectivity} $\ell^\infty(\X)$ is $\G$-injective and then we apply Theorem~\ref{coam}.

        \item Suppose  $C(\hh\X\backslash\hh\G)$ is an exact \cst-algebra. By Theorem~\ref{Relative Amenability and Injectivity Theorem} we have that $L^\infty(\hh\X\backslash\hh\G)$ is $\hh\G$-$W^*$-injective and relatively $\hh\G$-injective (where $L^\infty(\hh\X\backslash\hh\G)$ is relatively $\hh\G$-injective if $\Delta_{\hh\G}|_{L^\infty(\hh\X\backslash\hh\G)}$ admits a $\hh\G$-equivariant left inverse). Therefore, in the proof of Theorem~\ref{coam}, we can replace the nuclearity of $C(\hh\X\backslash\hh\G)$ by exactness and consider the nuclear map $\iota:C(\hh\X\backslash\hh\G)\to L^\infty(\hh\X\backslash\hh\G)$ that exists by injectivity of $L^\infty(\hh\X\backslash\hh\G)$.
    \end{enumerate}
\end{remark}

     A discrete quantum group $\G$ is {\it exact} if $C(\hh\G)$ is an exact \cst-algebra \cite{MR2355067}*{Proposition 1.28}. Therefore, if $\G$ is exact then all \cst-subalgebras of $C(\hh\G)$ are exact \cst-algebras. In the following corollary, since checking the exactness of a \cst-algebra $C(\hh{\G}_F\backslash\hh\G)$ is not straightforward, we replace this condition with exactness of $\G$.  
\begin{corollary}\label{wcFb}
    Let $\G$ be a discrete quantum group and $\alpha_F: \G\act  C(\fb)$. If $\G$ is unimodular or exact then $\Gamma_{\G_F} \prec\lambda_\G$.
\end{corollary}
\begin{proof}
The quantum group $\hh\G_F$ is a compact quantum subgroup of $\hh\G$ and $\ell^\infty(\G_F)$ is relatively amenable by \cite{MR4442841}*{Theorem 5.1}.
     Therefore, $\hh{\G}_F\backslash\hh\G$ is  coamenable  by Remark \ref{Relative Amenability Implies Coamenability Remark}, so the reduced version of $\Pi:C^u(\hh\G)\to C^u(\hh\G_F)$ exists \cite[]{MR4442841} and we denote it by $\tilde{\Pi}$.   We have  $\lambda_{\G_F}\circ\Pi= \tilde\Pi \circ \lambda_{\G}$, which implies   
  $\Gamma_{\G_F}=\lambda_{\G_F}\circ\Pi \prec\lambda_{\G}$.
\end{proof}

Our main results are stated next. The proofs follow easily from the constructions we have developed.
\begin{corollary}
    Let $\G$ be a unimodular discrete quantum group. If $C(\hh\G)$ has a unique trace then the action of $\G$ on its Furstenberg boundary is faithful.
\end{corollary}
\begin{proof}
    Using Corollary~\ref{wcFb}, we have that $\hh\G_F\backslash\hh\G$ is coamenable. Since $\hh\G_F\leq \hh\G$, and $\hh\G$ is Kac type, $\hh\G_F$ is Kac type. Therefore $\omega_{\hh\G_F}$ is tracial. By the unique trace property, $\omega_{\hh\G_F} = h_{\hh\G}$, hence $\hh\G_F = \hh\G$.
\end{proof}
A discrete quantum group $\G$ is called {\textit {\cst-simple}} if $C(\hh\G)$ is a simple \cst-algebra. The following theorem is a noncommutative version of the second claim in \cite{MR3735864}*{Theorem 1.3}.
\begin{theorem}\label{thm:faith-Fur-bnd-->unq-trc}
     Let $\G$ be a discrete quantum group that is either unimodular or exact. If $\G$ is \cst-simple then the action of $\G$ on its Furstenberg boundary is faithful.
\end{theorem}
\begin{proof}
By  Corollary \ref{wcFb}, $\Gamma_{\G_F} \prec\lambda_\G$. But $\G$ is \cst-simple so they are equivalent, $\Gamma_{\G_F} \sim\lambda_\G$, and $C(\hh\G)\cong C(\hh\G_F)$ which implies faithfulness of $\alpha$.
\end{proof}
Recall that a unimodular discrete quantum group $\G$ has the unique trace property if the Haar state $h_{\hh\G}$ is the unique trace on $C(\hh\G)$.  According to \cite{MR4442841}*{Theorem 5.3}, if a unimodular discrete quantum group has a faithful action on its Furstenberg boundary, then $\G$ has the unique trace property. Therefore, the most important consequence of Theorem \ref{thm:faith-Fur-bnd-->unq-trc} is as follows.
\begin{corollary}
    Let $\G$ be a \cst-simple  discrete quantum group. If $\G$ is unimodular, then $\G$ has the unique trace property. If $\G$ is not unimodular, then $C(\hh\G)$  does not have any KMS-state for the scaling automorphism group at the inverse temperature 1.
\end{corollary}
\begin{proof}
    The proof is based on Theorem \ref{thm:faith-Fur-bnd-->unq-trc} and \cite{MR4442841}*{Theorem 5.3}.
\end{proof}
\begin{example}
   The following examples are discrete quantum groups that are both exact and \cst-simple.  Thanks to our work, we know that they act faithfully on their Furstenberg boundaries.
    In the unimodular cases, it has been previously proven that the unique trace property holds.
     However, in the non-unimodular cases, they do not have any KMS-state for the scaling automorphism group when the temperature is at the inverse of 1.
    \begin{itemize}
        \item the unitary free quantum groups $\mathbb{F} U_Q$ (see \cite{MR1484551} for a definition and for the claims see \cite{MR1484551}*{Theorem 3} and \cite{MR3084500});
        \item the orthogonal free quantum groups $\mathbb{F}O_Q$ where $||Q||^8 \leq \frac{3}{8}\mathrm{tr}(QQ^*)$ (see \cite{MR2355067}*{Theorem 7.2}. (Note also that this result is not new (see below)));
        
        \item the duals of quantum automorphism groups $\widehat{\mathrm{QAut}(B, tr)}$ of a quantum space $(B, tr)$, where $B$ is a finite-dimensional \cst-algebra such that $dim(B) \geq 8$ and $tr$ is a $\delta$-trace (see \cite{MR3138849} for a definition and for the claims see \cite{MR3138849}*{Corollary 5.12} and \cite{MR3138849}*{Corollary 4.17});

        \item the duals of free wreath product quantum groups $\widehat{H_N^+(\Gamma)} := \widehat{\widehat{\Gamma}\wr S^+_N}$, where $\Gamma$ is an exact discrete group, $N\geq 8$, and $S^+_N$ is the quantum permutation group acting on $\C^N$ (see \cite{FT23} for a definition and for the claims see \cite{MR3250372}*{Theorem 3.5} and \cite{FT23}*{Theorem 4.1}).
    \end{itemize}
    We make two additional remarks regarding the above examples.
    \begin{itemize}
        \item It was proved with \cite{MR4442841}*{Corollary 7.14} that $\mathbb{F}O_Q$ acts faithfully on $C(\partial_F \mathbb{F}O_Q)$ whenever $Q\in M_N$ where $N\geq 3$. Most notably, this proves $C(O_Q^+)$ has the unique trace property (in the unimodular case) where it was previously only known when $||Q||^8 \leq \frac{3}{8}\mathrm{tr}(QQ^*)$.
        \item Brannan, Gao, and Weeks shared with us a preprint \cite{BGW} that contains the following result: let $\G = \widehat{\mathrm{QAut}(B, tr)}$ of a quantum space $(B, tr)$, where $B$ is a finite-dimensional \cst-algebra such that $dim(B) \geq 4$ and $tr$ is a $\delta$-trace. The action of $\G$ on $C(\partial_F\G)$ is faithful. Their proof is an adaptation of the proof of \cite{MR4442841}*{Corollary 7.14} and, most notably, proves that $C(\mathrm{QAut}(B, tr))$ has the unique trace property (it was previously only known when $dim(B) \geq 8$).
    \end{itemize}
\end{example}
Let $\alpha: A \to M(c_0(\G) \ot A)$ be a (left) action of a discrete quantum group $\G$ on a \cst-algebra $A$.
    A \emph{covariant representation} of $(A,\alpha)$ into a
    \cst-algebra $B(H_\pi)$ is a pair $(\rho,X_\pi)$ where $\rho : A \to
    B(H_\pi)$ is a non-degenerate $*$-homomorphism and $X_\pi \in \ell^\infty(\G)\otimes B(H_\pi)$ is a unitary corepresentation of $\G$ assigned to the $\G$-representation $\pi:\fO(\hh\G)\to B(H_\pi)$ satisfying the covariance relation
    $$(\id \ot \rho)\alpha(a) = X_\pi^*(1 \ot \rho(a))X_\pi \quad\text{for
        all}\quad a \in A \; .$$
        For any covariant representation $(\rho,X_\pi)$ of $(A,\alpha)$,
    the closed linear span of
    \[
    \rho(A) \{(\om \ot \id)(X_\pi) \mid \om \in
    \ell^\infty(\hh\G)_*\}
    \]
    is a \cst-algebra generated by $(\rho,X_\pi)$, denoted by $C^{\rho,\pi}(\hh\G)$ \cite{MR2355067}*{Definition 1.25, Proposition 1.26}.

 We emphasize that a universal version of \cite{MR4442841}*{Lemma 5.2} holds. In particular, if $\G$ is unimodular then a state on $C^u(\hh\G)$ is $\G$-invariant iff it is tracial. It is straightforward to then observe that a state on $C^\pi(\hh\G)$ for any $\G$-representation $\pi$ is tracial iff it is $\G$-invariant. In the next Theorem, we generalize \cite{MR4442841}*{Theorem 5.3} to arbitrary covariant representations $(\rho, X_\pi)$ where $C^\pi(\hh\G)$ admits a trace.  

\begin{theorem}\label{faith-Fur-bnd-->unq-trc}
    Let $\G$ be a unimodular discrete quantum group and assume that the action $\alpha:\G\act C(\fb)$   is faithful.  If there exists a covariant  representation $(\rho,X_\pi)$ of $(C(\fb),\alpha)$ such that the \cst-algebra $C^\pi(\hh\G)$ admits a trace $\tau$, then $\lambda_\G \prec\pi$. Moreover, $C^\pi(\hh\G)$ has a unique trace.
\end{theorem}
\begin{proof}    
    Since $\tau$ is a u.c.p. $\G$-equivariant map, one can define a u.c.p. $\G$-equivariant map $\tilde\tau:C^{\rho,\pi}(\hh\G)\to C(\fb)$ that extends $\tau$, using $\G$-injectivity of $C(\fb)$.
    
    Since the restriction of $\tilde\tau$ to $C(\fb)$ is the identity map so $\tilde\tau(\rho(a)\big( (\omega\otimes\id)X_\pi\big))=a \tau((\omega\otimes\id)X_\pi)$. Moreover, for every $\omega\in\ell^1(\G)$,
 \[
 \rho(a)\big( (\omega\otimes\id)X_\pi\big)=(\omega\otimes\id)(\I\otimes\rho(a))X_\pi=(\omega\otimes\id)\big(X_\pi (\id\otimes\rho)\alpha(a)\big),
 \]
where we used covariance relation for the last equality.     Applying $\tilde\tau$ on the above equations, for any $\mu\in C(\fb)^*$ we have
    $$\mu(a) (\omega\otimes\tau)X_\pi=(\omega\otimes\tau)X_\pi(\mu\ast a\otimes \I)$$
    In particular, setting $x = (\id\otimes\tau)(X_\pi)$ we have $\mu(a) x = x (\fP_\mu(a))$
    for all $\mu\in C(\fb)^*$ and $a\in C(\fb)$. Since $\mu(a) = \varepsilon(\fP_\mu(a))$ this yields $\varepsilon(y)x = xy$ for all $y\in \ell^\infty(\G_F)$. By faithfulness of $\alpha$, $\ell^\infty(\G_F)=\ell^\infty(\G)$, so for every $\omega\in C(\hh\G)^*$,
    \begin{align*}
    \omega(1)(\id\otimes\tau\circ\pi)(\ww_\G)&=\big((\id\otimes\tau\circ\pi)(\ww_\G)\big)\big((\id\otimes\omega)\ww_\G\big)=(\id\otimes\tau\circ\pi\otimes\omega)(\ww_\G)_{12}(\ww_\G)_{13}\\
    &=(\id\otimes\tau\circ\pi\otimes\omega)(\ww_\G)_{23}(\ww_\G)_{12}(\ww_\G)_{23}^*=(\id\otimes\tau\circ\pi\otimes\omega)(\id\otimes\Delta_{\hh\G})\ww_\G.
    \end{align*}
 Note that $C^u(\hh\G)^* \to \ell^\infty(\G), \mu\mapsto (\id\otimes \mu)(W_\G)$ is an injective homomorphism with dense range, hence $\varepsilon(y)x = xy$ implies $\omega*(\tau\circ\pi) = \omega(1)\tau\circ\pi$ for all $\omega\in C^u(\hh\G)^*$.
    This relation shows that $\tau\circ\pi$ is a $\hh\G$-invariant state on $\fO(\hh\G)$ so $\tau\circ\pi=h_{\hh\G}\circ\lambda_\G$ . Since  $h_{\hh\G}$ is faithful, we have $\lambda_\G \prec\pi$. Finally, If $\tau'$ is another trace on $C^\pi(\hh\G)$ then by the same argument $\tau'\circ\pi = h_{\hh\G}\circ \lambda_\G = \tau\circ\pi$.
\end{proof}

\subsection*{Acknowledgements}
B.A-S. would like to thank Roland Vergnioux and Lucas Hataishi for stimulating conversations on various aspects of this project. F.K. would like to express her gratitude to  Mehrdad Kalantar for numerous beneficial discussions and for his contribution in the initial stages of this project. F.K. is especially thankful to Roland Vergnioux and the Analysis group of the University of Caen Normandy for their warm hospitality and support during her visit in September 2022. Moreover, we would like to express our appreciation to the Fields Institute for Research in Mathematical Sciences, where we completed the final part of this work during the Thematic Program on Operator Algebras and Applications in September 2023. Lastly, we are grateful to the authors of \cite{BGW} for kindly sharing their preprint with us.

\end{document}